\documentclass[a4paper]{amsart}

\pdfoutput=1

\usepackage[utf8]{inputenc}
\usepackage[T1]{fontenc}
\usepackage{lmodern}
\usepackage{amsthm, amssymb, amsmath, amsfonts, mathrsfs}

\usepackage{microtype}

\usepackage[pagebackref,colorlinks=true,pdfpagemode=none,urlcolor=blue,
linkcolor=blue,citecolor=blue]{hyperref}

\usepackage{amsmath,amsfonts,amssymb,amsthm}
\usepackage{amsthm, amssymb, amsmath, amsfonts, mathrsfs}

\usepackage{mathrsfs}
\usepackage{MnSymbol}
\usepackage{color}
\usepackage{accents}

\usepackage{bbm}

\usepackage{cleveref}

\newtheorem{theorem}{Theorem}[section]
\newtheorem{lemma}[theorem]{Lemma}

\newtheorem{corollary}[theorem]{Corollary}
\newtheorem{proposition}[theorem]{Proposition}

\theoremstyle{remark}
\newtheorem{remark}[theorem]{Remark}

\renewenvironment{proof}[1][Proof]{ {\itshape \noindent {#1.}} }{$\Box$
\medskip}

\numberwithin{equation}{section}
\newcommand{\R}{\mathbb{R}}

\newcommand{\Z}{\mathbb{Z}}
\newcommand{\N}{\mathbb{N}}
\newcommand{\Pb}{\mathbb{P}}

\newcommand{\E}{\mathbb{E}}

\newcommand{\C}{\mathcal{C}}

\newcommand{\eps}{\varepsilon}

\newcommand{\la}{\langle}
\newcommand{\ra}{\rangle}

\newcommand{\X}{\mathcal{X}}

\newcommand{\1}{\mathbbm{1}}

\newcommand{\bfE}{\mathbf{E}}
\newcommand{\cL}{\mathcal{L}}
\newcommand{\rmF}{\mathrm{F}}
\newcommand{\x}{\mathbf{x}}

\definecolor{darkmagenta}{rgb}{1,0,1}

\newcommand{\gu}[1]{\textcolor{black}{#1}}

\newcommand{\be}{\begin{equation}}
\newcommand{\ee}{\end{equation}}

\begin{document}

\title{A PDE hierarchy for directed polymers in random environments}
\author{Yu Gu, Christopher Henderson}

\address[Yu Gu]{Department of Mathematics, Carnegie Mellon University, Pittsburgh, PA 15213}

\address[Christopher Henderson]{Department of Mathematics, University of Arizona, Tucson, AZ 85721}

\maketitle

\begin{abstract}
For a Brownian directed polymer in a Gaussian random environment, with $q(t,\cdot)$ denoting the quenched endpoint density and  \[
Q_n(t,x_1,\ldots,x_n)=\bfE[q(t,x_1)\ldots q(t,x_n)],
\] we derive a hierarchical PDE system satisfied by $\{Q_n\}_{n\geq1}$. We present two applications of the system:  (i) \gu{we compute the generator of $\{\mu_t(dx)=q(t,x)dx\}_{t\geq0}$ for some special   functionals, where $\{\mu_t(dx)\}_{t\geq0}$ is viewed as a Markov process taking values in the space of probability measures;} (ii) in the high temperature regime with $d\geq 3$, we prove a quantitative central limit theorem for the annealed endpoint distribution of the diffusively rescaled polymer path. \gu{We also study a nonlocal diffusion-reaction equation motivated by the generator and establish a super-diffusive $O(t^{2/3})$ scaling.}


%
%
%

\end{abstract}
\maketitle

\section{Introduction}



The study of directed polymers in random environments has witnessed important progress in recent years. A common feature of the  models is an interaction between a reference path measure, which is typically given by a random walk or a Brownian motion, and a background random environment. The polymer measure is then formulated as the Gibbs measure with a Hamiltonian describing the accumulated energy collected along the path in the random environment. The physically interesting quantities include the fluctuations of the free energy, the typical behaviors of the paths, and so on, see e.g. the books \cite{comets2017directed,den2009random,giacomin2007random} and the references therein.

While the random walk/Brownian motion is diffusive, the random environment can change the polymer's behavior drastically. Indeed, when the temperature is low, the typical path is expected to be super-diffusive, and the transverse displacement of the path of length $T$ is expected to be of order $T^{\xi}$ with an exponent $\xi>\tfrac12$. In $d=1$, it has been conjectured that $\xi=\tfrac23$ and the directed polymer model falls into the KPZ universality class.  
The proofs of this conjecture in several settings can be found, for example, in
\cite{amir2011probability,balazs2011fluctuation,barraquand2017random,borodin2014macdonald,borodin2014free,borodin2013log,corwin2016kpz,ferrari2006scaling,johansson2000shape,landim2004superdiffusivity,quastel2007t1,seppalainen2012scaling}. For a more complete list, we refer to the review articles \cite{corwin2012kardar,quastel2015one}. Another feature of the polymer paths in low temperatures is that they localize and concentrate in small regions, see e.g. \cite{bates2016endpoint,carmona2006strong,comets2003directed,vargas2007strong} and the references therein.

In this paper, we consider the Wiener measure as the reference path measure and a generalized spacetime Gaussian random field as the environment, and our focus will be on the endpoint distribution of the polymer path. 
The results presented in the sequel offer a possibly new perspective of this problem.  We develop a deterministic hierarchy governing the evolution of the endpoint distribution. \gu{From the hierarchy, we compute the generator of the Markov process associated with the endpoint distribution, for linear   functionals.} 
In $d\geq 3$ and for high temperatures, we make use of the hierarchy to quantify the diffusive behavior of the polymer path. \gu{We also study a nonlocal reaction-diffusion equation motivated by the generator, and establish a super-diffusive behavior with exponent $\tfrac23$; that is, we establish spreading on spatial scales $O(t^{2/3})$.} Our arguments are based on a mixture of tools from both probability and partial differential equations.

\subsection{Main result}


Let $\eta$ be a spacetime white noise  built on the probability space $(\Omega,\mathcal{F},\mathbf{P})$ and the expectation with respect to $\eta$ is denoted by $\bfE$. Fix a mollifier $0 \leq \phi\in \C_c^{\infty}(\R^d)$ with $\int \phi=1$.  We smooth $\eta$ in the $x-$variable and define
\[
\eta_\phi(t,x)=\int_{\R^d}\phi(x-y)\eta(t,y)dy,
\]
which is a generalized Gaussian random field. The covariance function of $\eta_\phi$ is then given by
\[
\bfE[\eta_\phi(t,x)\eta_\phi(s,y)]=\delta(t-s)R(x-y),
\]
with the spatial covariance  function
\[
R(x)=\int_{\R^d} \phi(x+y)\phi(y)dy \in C_c^\infty(\R^d).
\]

Let $(\Sigma,\mathcal{A},\Pb)$ be another probability space and $w$ be a Brownian motion built on it. The initial location $w_0$ is distributed according to $\mu_0(dx)$. Throughout the paper, we assume 
\[
\mu_0(dx)=q_0(x)dx, \quad\quad\quad\quad   q_0\geq 0, \quad \int_{\R^d} q_0(x)dx=1,
\] 
and consider the two cases (i) $q_0\in C_c(\R^d)$ (ii) $q_0(x)=\delta(x)$. While other initial distributions can be considered as well, our main focus here is on $q_0(x)=\delta(x)$. We denote the expectation with respect to $w$ by $\E_{\mu_0}$, and define the energy of any Brownian path $w:[0,T]\to\R^d$ in the Gaussian environment $\eta_\phi$ by
\begin{equation}\label{e.defHt}
H(T,w)=\int_0^T \eta_\phi(t,w_t)dt.
\end{equation}
 In this paper, we study the endpoint distribution of the random polymer, obtained by tilting the Brownian motion by a factor of $e^{\beta H(T,w)}$, where $\beta>0$ is the inverse temperature; that is, for any $T\geq 0$, we are interested in:
 \begin{equation}\label{e.endpoint}
\begin{aligned}
&\mu_T(dx)=q(T,x)dx, \\
&q(T,x)=Z(T)^{-1}\E_{\mu_0}[\delta(w_T-x)\exp(\beta H(T,w)-\tfrac12\beta^2R(0)T)],\ \ \text{ and}\\
&Z(T)=\E_{\mu_0}[\exp(\beta H(T,w)-\tfrac12\beta^2R(0)T)].
\end{aligned}
\end{equation}

In $d=1$, we also consider the random environment given by the spacetime white noise $\eta$, without any mollification in the spatial variable. To unify the notation, we allow $\phi$ to be the Dirac function $\phi(x)=\delta(x)$ in $d=1$, and the spatial covariance function in this case is
\[
R(x)=\int_{\R} \phi(x+y)\phi(y)dy=\delta(x).
\]
We define
\begin{equation}\label{e.endpointwhite}
\begin{aligned}
&q(T,x)=Z(T)^{-1}\E_{\mu_0}[\delta(w_T-x):\exp(\beta\textstyle\int_0^T \eta(s,w_s)ds):],\\
&Z(T)=\E_{\mu_0}[:\exp(\beta \textstyle\int_0^T \eta(s,w_s)ds):].
\end{aligned}
\end{equation}
Here $:\exp:$ is the Wick-ordered exponential, see e.g. \cite{nualart1997weighted}. In both \eqref{e.endpoint} and \eqref{e.endpointwhite}, the endpoint density $q$ is related to a stochastic heat equation with a multiplicative noise, which we will make more precise in Section~\ref{s.poshe}. We emphasize that the case of $\phi(x)=R(x)=\delta(x)$ is restricted to $d=1$.

For any $t\geq0, n\geq 1$ and $\mathbf{ x}_{1:n}:=(x_1,\ldots,x_n)\in\R^{nd}$, define the $n-$point density by 
\begin{equation}\label{e.ndensity}
Q_n(t,\mathbf{x}_{1:n})=\bfE[q(t,x_1)\ldots q(t,x_n)].
\end{equation}
For any two functions $f,g$, we define $\la f,g\ra:=\int f(x)g(x)dx$  as long as the integral is well-defined in a standard way. 
If $f$ and $g$ also depend on the $t$ variable, then we write $\la f(t),g(t)\ra=\int f(t,x)g(t,x)dx$; see \Cref{s.notation} for more details on the notations and conventions. We are now able to state the first main result.
\begin{theorem}\label{t.bbgky}
For any $n\geq1$ and $T>0$, if $f\in C_b^{1,2}([0,T]\times \R^{nd})$, we have
\begin{equation}\label{e.eqP}
\begin{aligned}
\la f(T), Q_n(T)\ra=\la f(0), q_0^{\otimes n}\ra&+\textstyle\int_0^T \la (\partial_t+\tfrac12\Delta) f(t), Q_n(t)\ra dt\\
&+\beta^2\sum_{k=0}^2\int_0^T\la f_{k,R}(t),Q_{n+k}(t) \ra dt,
\end{aligned}
\end{equation}
where the functions $f_{k,R}: [0,T]\times\R^{(n+k)d}\to \R$ are given by
\begin{equation}\label{e.defFR}
\begin{aligned}
&f_{0,R}(t,\x_{1:n})=f(t,\x_{1:n})\textstyle\sum_{1\leq i<j\leq n}R(x_i-x_j),\\
&f_{1,R}(t,\x_{1:n},x_{n+1})=-nf(t,\x_{1:n})\textstyle\sum_{i=1}^n R(x_i-x_{n+1}),\\
&f_{2,R}(t,\x_{1:n},x_{n+1},x_{n+2})=\tfrac12n(n+1)f(t,\x_{1:n})R(x_{n+1}-x_{n+2}).
\end{aligned}
\end{equation}
In other words, $\{Q_n\}_{n\geq 1}$ is a weak solution to the following hierarchy:
\begin{equation}\label{e.hierarchy}
\begin{split}
\partial_tQ_n&(t,\mathbf{x}_{1:n})=\tfrac12\Delta Q_n(t,\x_{1:n})+\beta^2Q_n(t,\x_{1:n})\textstyle\sum_{1\leq i<j\leq n}R(x_i-x_j)\\
&-\beta^2n\textstyle\int_{\R^d}Q_{n+1}(t,\x_{1:n},x_{n+1})\sum_{i=1}^nR(x_i-x_{n+1})dx_{n+1}\\
&+\beta^2\tfrac{n(n+1)}{2}\textstyle\int_{\R^{2d}} Q_{n+2}(t,\x_{1:n},x_{n+1},x_{n+2})R(x_{n+1}-x_{n+2})dx_{n+1}dx_{n+2},
\end{split}
\end{equation}
with the initial condition $Q_n(0,\cdot)=q_0^{\otimes n}$.
\end{theorem}

\subsection{Applications of the PDE hierarchy}

Let $\mathcal{M}_1(\R^d)$ be the space of probability measures on $\R^d$. Due to white-in-time correlation of $\eta$, $\{\mu_T\}_{T\geq0}$ is a Markov process taking values in $\mathcal{M}_1(\R^d)$. For any $f\in C_b(\R^d)$, we associate it in the natural way with a functional $\mathrm{F}_f:\mathcal{M}_1(\R^d)\to \R$ given by
\[
\mathrm{F}_f(\mu)=\la f,\mu\ra =\int_{\R^d} f(x) \mu(dx),
\]
where we abused the notation to  also let $\la\cdot,\cdot\ra$ denote the pairing between $C_b(\R^d)$ and $\mathcal{M}_1(\R^d)$. Denote the generator of $\{\mu_T\}_{T\geq0}$ by $\mathcal{L}$, and let $\star$ denote the convolution. An immediate consequence of Theorem~\ref{t.bbgky} is 
\begin{corollary}\label{c.generator}
Assume $\mu_0(dx)=q_0(x)dx$ with $q_0\in C_c(\R^d)$. For any $f\in C_b^2(\R^d)$,
\begin{equation}\label{e.generatorcolor}
\cL \rmF_f(\mu_0)=\la \tfrac12\Delta f,q_0\ra+\beta^2\la f, \mathcal{T} q_0\ra, 
\end{equation}
with the operator $\mathcal{T}$ defined as
\begin{equation}\label{e.defLR}
\begin{aligned}
\mathcal{T} q_0(x)= \la R\star q_0,q_0\ra q_0(x)-q_0(x)R\star q_0(x).
\end{aligned}
\end{equation}
\end{corollary}

Another application of Theorem~\ref{t.bbgky} is to study the diffusive behavior of the polymer endpoint in a high temperature regime when $d\geq3$. It is well-known from the classical work \cite{albeverio,bolt,commets,spencer,Song} that in this case and under a diffusive rescaling, the polymer endpoint converges to a standard normal distribution in the quenched sense. In our notation, as $q(T,\cdot)$ denotes the quenched density of the endpoint of the polymer of length $T$, 
 the result says that for almost every realization of the random environment, and any $h\in C_b(\R^d)$, we have
\begin{equation}\label{e.qinv}
\int_{\R^d} h(x) T^{d/2}q(T,\sqrt{T}x) dx\to \int_{\R^d} h(x) G_1(x)dx, \quad\quad \mbox{ as }T\to\infty.
\end{equation}
Here $G_1(\cdot)$ is the density of 
$N(0,\mathrm{I}_d)$. The results in \cite{bolt,commets,spencer} are for a discrete i.i.d. random environment. In the setting of the continuous Gaussian environment considered in this paper, the same result was proved in \cite{chiranjib}. 
The above quenched central limit theorem~\eqref{e.qinv} immediately implies the annealed one (recall that $Q_1=\E[q]$)
 \begin{equation}\label{e.ainv}
 \int_{\R^d} h(x) T^{d/2}Q_1(T,\sqrt{T}x) dx\to \int_{\R^d} h(x) G_1(x)dx, \quad\quad \mbox{ as }T\to\infty.
 \end{equation}
As $\{Q_n\}_{n\geq1}$ solves the PDE hierarchy \eqref{e.hierarchy}, which can be viewed as a ``perturbation'' of the heat equation, it is natural to ask if we can analyze the system of equations and show that the ``perturbation'' is indeed small in this asymptotic regime. It turns out that the hierarchy  provides a nice analytic framework for us to give a simple proof of \eqref{e.ainv} and to also quantify  the convergence rate.
  
Let $X_T$ denote a random variable with the density $T^{d/2}Q_1(T,\sqrt{T}x)$, then the Wasserstein distance between $X_T$ and $N(0,\mathrm{I}_d)$ is defined as 
 \[
 d_{\mathrm{W}}(X_T,N(0,\mathrm{I}_d)):=\sup_{h\in\mathrm{Lip}(1)}\left|\int_{\R^d} h(x)T^{d/2}Q_1(T,\sqrt{T}x) dx-\int_{\R^d} h(x) G_1(x)dx\right|,
 \]
 where $\mathrm{Lip}(1)=\{h\in C(\R^d): |h(x)-h(y)|\leq |x-y|\}$.
  \begin{theorem}\label{t.qclt}
 Assume $q_0(\cdot)=\delta(\cdot)$. In $d\geq3$, there exist positive constants $\beta_0(d,R),C(d,R,\beta)$ such that if $\beta<\beta_0$, we have for all $T>0$ that
 \begin{equation}\label{e.wdbd}
 \begin{aligned}
	 d_{\mathrm{W}}(X_T,N(0,\mathrm{I}_d))
		 \leq  C\left(\tfrac{\log T}{\sqrt{T}}\1_{d=3}+\tfrac{1}{\sqrt{T}}\1_{d\geq 4}\right).
 \end{aligned}
 \end{equation}
 \end{theorem}
 
 \gu{Similar results were obtained in \cite[Theorem 4]{bold}, and these suggest that the error estimates above are sharp when $d\geq 5$.}
 
 Of particular interest is the mean square displacement of the polymer endpoint. We have the following error bound.
 \begin{theorem}\label{t.msd}
 Under the same assumption of Theorem~\ref{t.qclt}, it holds for all $T>0$ that
 \begin{equation}\label{e.msdbd}
  \big|\tfrac{1}{T}\int_{\R^d} |x|^2 Q_1(T,x)dx-d\big|\leq C\left(\tfrac{1}{\sqrt{T}}\1_{d=3}+\tfrac{\log T}{T}\1_{d=4}+\tfrac{1}{T}\1_{d\geq5}\right).
  \end{equation}
 \end{theorem}

There are several results on error estimates for the mean square displacement \cite{albeverio,spencer,Song}, and Theorem~\ref{t.msd} seems to provide the best rate. We discuss the relation of Theorem~\ref{t.msd} to the previous results in more details in Remark~\ref{r.err} below.

\gu{\subsection{A nonlocal reaction-diffusion equation motivated by $\mathcal{L}$}
Since $\{\mu_T\}_{T\geq0}$ is a Markov process, we can consider the forward Kolmogorov equation associated with it. Recall that $\mathcal{L}$ denotes its generator. For any $f\in C_b(\R^d)$, let 
\[
\mathscr{U}:[0,\infty)\times \mathcal{M}_1(\R^d)\to \R
\]
 be the solution to 
\begin{equation}\label{e.abseq}
\begin{aligned}
&\partial_t \mathscr{U}(t,\mu)=\cL \mathscr{U}(t,\mu), \quad \quad t>0,\\
&\mathscr{U}(0,\mu)=\la f,\mu\ra.
\end{aligned}
\end{equation}
With $\mu_0(dx)=q_0(x)dx$, the solution can be written as
\[
\mathscr{U}(t,\mu_0)=\bfE[\la f,\mu_t\ra]=\int_{\R} f(x)Q_1(t,x)dx,
\]
where we used the fact that $\mu_t(dx)=q(t,x)dx$ and $\bfE[q(t,x)]=Q_1(t,x)$. Thus, one can try to study the asymptotic behavior of $Q_1(t,\cdot)$ as $t\to\infty$ by considering the equation \eqref{e.abseq}, rather than the system \eqref{e.bbgky}.}

\gu{Nevertheless, computing the generator $\mathcal{L}$ for general  functionals is difficult, and, in addition, the obtained formula might not be easy to manipulate. In Corollary~\ref{c.generator}, we considered the simplest possible linear   functionals, and the equation \eqref{e.generatorcolor} shows that, for these functionals on $\mathcal{M}_1(\R^d)$, the action of $\cL$ is equivalent to that of the differential operator $\tfrac12\Delta +\beta^2\mathcal{T}$ acting on the corresponding density. This motivates us to consider the following deterministic PDE associated with the operator $\mathcal{T}$. Take the case of $R(\cdot)=\delta(\cdot)$ in $d=1$, and denote the $L^2(\R^d)$ norm by $\|\cdot\|$. For any $f\in \gu{L^2(\R)}$, we have
\[
\mathcal{T} f(x)= \|f\|^2 f(x)-f(x)^2.
\]
Consider the following equation 
\begin{equation}\label{e.maineq}
\begin{aligned}
\partial_t g(t,x)&=\tfrac12\Delta g(t,x)+\beta^2\mathcal{T} g(t,\cdot)\\
&=\tfrac12\Delta g(t,x)+\beta^2\|g(t,\cdot)\|^2 g(t,x)-\beta^2g(t,x)^2, \quad\quad t>0, x\in\R,\\
g(0,x)&=q_0(x).
\end{aligned}
\end{equation}
We see in the sequel that this describes the evolution of a probability density.}

\gu{It is unclear whether \eqref{e.maineq} has anything to do with \eqref{e.abseq}, since we only know that the two evolutions match at $t=0$, which is due to the fact that $\mathscr{U}(0,\mu)=\la f,\mu\ra$ is a linear functional on $\mathcal{M}_1(\R)$. Nevertheless, the following result shows a super-diffusive behavior of $g$ 
with the exponent $\tfrac23$.
\begin{theorem}\label{t.23}
In $d=1$, assume $0\leq\,  q_0\in C_c(\R)$ and $\int_{\R} q_0(x)dx = 1$. For any $p\geq1$, there exists a constant $C=C(p,\beta,q_0)>0$ such that  
\[
C^{-1}T^{\frac{2p}{3}}\leq \int_{\R}|x|^pg(T,x)dx \leq C T^{\frac{2p}{3}}, \quad\quad \mbox{ for all } T\geq1.
\] 
\end{theorem}}

\gu{It is unclear to us whether the above super-diffusion with the ``right'' exponent $\tfrac{2}{3}$ is a coincidence or not. In the case of the spacetime white noise environment in $d=1$,  the function $Q_1$ is the annealed density of the endpoint of the continuum directed random polymer \cite{alberts2014continuum}, and $\{Q_n\}_{n\geq 1}$ is a weak solution of the infinite system:
\begin{equation}\label{e.bbgky}
\begin{aligned}
\partial_t Q_n(t,x_1,\ldots,x_n)=&\tfrac12\Delta Q_n(t,x_1,\ldots,x_n)\\
&+\beta^2Q_n(t,x_1,\ldots,x_n)\textstyle\sum_{1\leq i<j\leq n}\delta(x_i-x_j)\\
&-\beta^2n\textstyle\sum_{i=1}^nQ_{n+1}(x_1,\ldots,x_n,x_i)\\
&+\beta^2\tfrac{n(n+1)}{2}\textstyle\int_{\R}Q_{n+2}(x_1,\ldots,x_n,\tilde{x},\tilde{x})d\tilde{x}, \quad\quad n\geq1.
\end{aligned}
\end{equation}
For this model, the super-diffusion with exponent $\tfrac{2}{3}$ was shown in \cite[Theorem 1.11]{corwin2016kpz}. Taking $n=1$, the above system yields
\begin{equation}\label{e.bbgky1}
\partial_t Q_1(t,x)=\tfrac12\Delta Q_1(t,x)+\beta^2\textstyle\int_{\R} Q_3(t,x,\tilde{x},\tilde{x})d\tilde{x}-\beta^2Q_2(t,x,x).
\end{equation}
Thus, the question reduces to justifying if \eqref{e.maineq} is a reasonable approximation of \eqref{e.bbgky1} or not.
Incidentally, if we make the assumption of a factorized joint density to close the hierarchy
\begin{equation}\label{e.ass}
Q_2(t,x_1,x_2)\approx Q_1(t,x_1)Q_1(t,x_2), \quad \quad Q_3(t,x_1,x_2,x_3)\approx \prod_{j=1}^3 Q_1(t,x_j),
\end{equation}
which is similar in spirit to the molecular chaos assumption in the BBGKY hierarchy of kinetic theory \cite{cercignani2013mathematical}, then \eqref{e.bbgky1} reduces to \eqref{e.maineq}:
\begin{equation}\label{e.ass1}
\partial_t Q_1(t,x)\approx \tfrac12\Delta Q_1(t,x)+\beta^2\|Q_1(t,\cdot)\|^2Q_1(t,x)-\beta^2 Q_1^2(t,x).
\end{equation}
However, we do not claim either \eqref{e.ass} or \eqref{e.ass1} in this paper.}

\subsection{Discussions}

As the endpoint density of the Brownian motion solves the standard heat equation, we look for a counterpart when the Brownian motion is weighted by a random environment. For each fixed realization of the random environment, it is known that the polymer model is equivalent to a diffusion in a (different) random environment \cite[Theorem 2]{bakhtin2018global}. Thus, in the quenched setting, the analogue of the standard heat equation we are looking for is a Fokker-Planck equation with a random coefficient, describing the evolution of the density of the aforementioned diffusion. However, studying either the solution to the Fokker-Planck equation or its ensemble average seems to be as complicated as the polymer model itself; hence, the main message we wish to convey here is the following: rather than studying the single point distribution, one could instead look at the multipoint distributions defined in \eqref{e.ndensity}. By definition, for each $T\geq 0$, $Q_n(T,\cdot)$ is a  probability density on $\R^{nd}$. While we do not have an underlying dynamics that reproduces the evolution of $Q_n$, heuristically, it can be viewed as the joint density of $n$ particles, interacting indirectly through their separate individual interaction with the common random environment, similar to ``independent walkers in the same environment''.  Theorem~\ref{t.bbgky}, which comes from a straightforward application of It\^o's formula, shows that $\{Q_n\}_{n\geq1}$ solves a hierarchical PDE system. In this way, the study of the endpoint distribution of the random polymer, in the annealed setting, may be reduced to the study of $Q_1$ and the analysis of the deterministic PDE system satisfied by $\{Q_n\}_{n\geq1}$.  

We make a few remarks.

\begin{remark}
A result similar to Theorem~\ref{t.bbgky} was proved in \cite[Theorem 3.1]{carmona2006strong} for a different polymer model (albeit not formulated as a PDE hierarchy), with a random walk reference path measure and an environment on $\R_+\times \Z^d$ made of i.i.d. copies of Brownian motions.
\end{remark}

\begin{remark}
We show in the sequel (see \Cref{l.timereversal}) that
\begin{equation}\label{e.numerator}
Q_n(t,x_1,\ldots,x_n)=\bfE\bigg[\frac{u(t,x_1)\ldots u(t,x_n)}{(\int_{\R} u(t,x)dx)^n}\bigg],
\end{equation}
with $u$ solving the stochastic heat equation 
\[
\partial_t u=\tfrac12\Delta u+\beta u\eta_\phi, \quad\quad u(0,\cdot)=q_0(\cdot).
\]
If we only consider the numerator of the r.h.s.~of \eqref{e.numerator}, and define 
\[
\tilde{Q}_n(t,x_1,\ldots,x_n)=\bfE[u(t,x_1)\ldots u(t,x_n)],
\] it is well-known that $\tilde{Q}_n$ solves 
\begin{equation}\label{e.tildeQ}
\partial_t \tilde{Q}_n=\tfrac12\Delta \tilde{Q}_n+\beta^2\tilde{Q}_n\textstyle\sum_{1\leq i<j\leq n} R(x_i-x_j)=:\mathscr{H}_n\tilde{Q}_n.
\end{equation}
In this case there is no coupling between $\tilde{Q}_n$ for different values of $n$, and the Hamiltonian $\mathscr{H}_n$ is the so-called Delta-Bose gas if we have a contact interaction $R(\cdot)=\delta(\cdot)$ in $d=1$. There are many studies on the moments of the stochastic heat equation, either relying on the Feynman-Kac representation of the solution to \eqref{e.tildeQ} or the spectral property of $\mathscr{H}_n$, and we refer to the monograph \cite{davar} and the references therein. The equation \eqref{e.tildeQ} should be compared with \eqref{e.hierarchy}, in which we have additional terms related to $Q_{n+1}$ and $Q_{n+2}$.
\end{remark}

\begin{remark}
For the \emph{annealed} endpoint distribution considered in the present paper, with the density given by $\bfE[q(T,\cdot)]$, it is conjectured that in $d=1$, the rescaled density $T^{\frac23}\bfE[q(T,T^{\frac23}\cdot)]$ converges weakly in space to some universal limit as $T\to\infty$, \gu{see the result on a related last passage percolation model \cite{johansson2003discrete}.} The limit was further identified in \cite{flores2013endpoint,schehr2012extremes}.
\end{remark}

\gu{\begin{remark}
The asymptotics of the solution of~\eqref{e.maineq} are not obvious. In order to understand the exponent $\frac23$ in Theorem~\ref{t.23}, one can make the following back-of-the-envelope computation.  If we assume spreading at spatial scales $O(t^p)$, then to preserve the fact that $g$ is a probability measure, we must have $\sup g \sim O(t^{-p})$.  This yields $\|g\|^2 \sim O(t^{-p})$.  In order to use this, we linearize the equation around zero to obtain
\[
	\partial_t g \approx \frac{1}{2} \Delta g + \beta^2 \|g\|^2 g
		\approx \frac{1}{2} \Delta g + O(t^{-p}) g.
\]
Then, using the large $x$ asymptotics of the heat equation and the fact that the last term yields an integrating factor, we find
\[
	g(t,x)
		\approx e^{\int_0^t O(s^{-p}) ds - \frac{x^2}{2t} + O(\log(t))}
		\approx e^{O(t^{1-p}) - \frac{x^2}{2t}}.
\]
For consistency with our assumption of spreading in $x$ like $O(t^p)$, we require that $g$ is ``large'' to the left of $O(t^p)$ and ``small'' to the right of $O(t^p)$.  This means that the two terms in the exponent should cancel at $x \sim O(t^p)$.  In other words, we require $O(t^{1-p}) = (t^p)^2/2t$.  Solving this yields $p=\frac23$. Unfortunately, this argument is far from rigorous.   Instead, as with the heat equation, in order to establish the spreading behavior of $g$, the key estimate is an upper bound on the $L^\infty$ norm that yields decay to zero at the sharp rate, which is $O(t^{-\frac23})$.   Since the Laplacian (diffusion) can only cause decay like $O(t^{-\frac12})$ in $d=1$, the nonlinear terms have to provide the mechanism for this decay.  Our proof proceeds by establishing a functional inequality relating the two nonlinear terms at any maximum of $g$.  \gu{This, combined with} a differential inequality satisfied by the maximum, shows that $g(t) \lesssim t^{-\frac23}$.  From there we obtain the upper bound in \Cref{t.23} via the construction of a supersolution and the lower bound via a simple variational argument.
\end{remark}}

\gu{\begin{remark}
It is natural to investigate further questions on the PDE \eqref{e.maineq}, including the asymptotics of $T^{\frac{2}{3}}g(T,T^{\frac23}\cdot)$ as $T\to\infty$ and the behavior of $g$ in high dimensions. One can study the equation corresponding to the spatially correlated noise:
\[
\partial_t g(t,x)=\tfrac12\Delta g(t,x)+\beta^2\la R\star g(t,\cdot),g(t,\cdot)\ra g(t,x)-\beta^2g(t,x)R\star g(t,x),
\]
where $R\in C_c^\infty(\R^d)$ is the spatial covariance function of the noise. Compared to \eqref{e.maineq} which is the case of $R(\cdot)=\delta(\cdot)$, the above equation is ``more nonlocal'', 
and the analytic tools used in this paper do not seem to apply.  Indeed, the functional inequalities and delicate identities used in the proof of \Cref{t.23} either are not true or do not make sense and do not have obvious analogues in this more general setting. These are presented in \cite{gh}.
\end{remark}}

\begin{remark}
We note that  the proof of \Cref{t.23} does not use any regularity of $q_0$, and would apply equally well to $q_0$ that is a localized probability measure such as $\delta$; however, it is not immediately obvious that~\eqref{e.maineq} is well-posed with measure initial data.  Hence, in this work, we impose the condition that $q_0 \in C_c(\R)$ in order to avoid technical issues.  In a future work, we show that this condition may be relaxed; that is, the estimates established here are sufficient to establish such a well-posedness result for localized probability measures.
\end{remark}

\begin{remark}
In $d\geq 3$ and the high temperature regime, Theorems~\ref{t.qclt} and \ref{t.msd} concern the annealed distribution $Q_1$, which is obtained after taking an average with respect to the random environment. It is natural to ask about the extra error induced by the random fluctuations of the environment. As the focus of the paper is on the PDE hierarchy, which involves $Q_1$ rather than $q$, we do not study this problem here,  but note that it is not very hard to extract error bounds on the random fluctuations from the proof of Theorem~\ref{t.bbgky}. \gu{We refer to Remark~\ref{r.781} below for more details.}
\end{remark}

\begin{remark}
In Theorems~\ref{t.qclt} and \ref{t.msd}, we are deep in the high temperature regime $\beta\ll1$ to obtain the error estimate in the annealed central limit theorem. The quenched central limit theorem actually holds in the full weak disorder regime $\beta<\beta_c$ for some critical value of  $\beta_c$ \cite{commets}. As we strive for a more quantitative estimate here, it is unclear to us whether similar error estimates can be proved for all $\beta<\beta_c$.
\end{remark}

\begin{remark}\label{r.err}
The first error bound on the convergence of the mean square displacement was proved in \cite[Eq (1.7)]{spencer}, and their result in the discrete setting translates to our case as
\[
\big|\tfrac{1}{T}\int_{\R^d} |x|^2 Q_1(T,x)dx-d\big|\leq C\tfrac{1}{T^{\theta-\delta}}.
\]
Here $\theta=\min(\tfrac{d-2}{4},\tfrac{3}{4})$ and $\delta>0$ can be arbitrarily small.  In \cite{albeverio,Song}, similar results were shown which corresponds to  the following in our setting
\[
\big|\tfrac{1}{T}\int_{\R^d} |x|^2 Q_1(T,x)dx-d\big| \leq C \Big( \tfrac{1}{T^{(d-2)/4}}\1_{2<d<6}+\tfrac{\sqrt{\log T}}{T}\1_{d=6}+\tfrac{1}{T}\1_{d>6}\Big).
\]
It is worth mentioning that in these works, the quenched error estimates were also proved, while we only focus on the annealed case here.
\end{remark}

\subsubsection*{Organization of the paper}
In Section~\ref{s.poshe}, we briefly discuss the connection between the directed polymer and the stochastic  heat equation, which will be used later in our proof. Sections~\ref{s.bbgky}, \ref{s.qclt} and  \ref{s.23} are devoted to the proofs of Theorem~\ref{t.bbgky},  \ref{t.qclt} and \ref{t.23} respectively. In Appendix~\ref{s.she}, we review some basics about stochastic heat equations for the convenience of readers. The proofs of some technical lemmas are presented in Appendix~\ref{s.lem}.

\subsection{Notation and conventions} \label{s.notation}
We recall and define some notation.

(i) The expectation with respect to the Gaussian random environment is denoted by $\mathbf{E}$, and the expectation with respect to Brownian motions is $\E$.

(ii) We consider two cases of spatial covariance functions of the Gaussian environment (a) $R(\cdot)\in C_c^\infty(\R^d), d\geq 1$ and (b)  $R(\cdot)=\delta(\cdot), d=1$.

(iii) The initial distribution $\mu_0(dx)=q_0(x)dx$ is fixed, and we include the two cases (a) $q_0(\cdot)\in C_c(\R^d)$ and (b) $q_0(\cdot)=\delta(\cdot)$.

(iv) For functions $f,g$ and measure $\mu$, we write $\la f,g\ra=\int fg$, $\la f,\mu\ra=\int f(x)\mu(dx)$, and $\|f\|^2=\la f,f\ra$.

(v) We use $\star$ to denote convolution in the spatial variable $x$, and the standard heat kernel of $\partial_t -\tfrac12\Delta$ is $G_t(x)=(2\pi t)^{-d/2}\exp(-|x|^2/(2t))$.


\subsubsection*{Acknowledgement}

YG was partially supported by the NSF through DMS-1907928 and the Center for Nonlinear Analysis of CMU. CH was partially supported by NSF grant DMS-2003110. 
YG would like to thank Xi Geng for several discussions.

\section{Directed polymer and stochastic heat equation}
\label{s.poshe}

In this section, we briefly discuss the relationship between directed polymers and the stochastic heat equation with a multiplicative noise. 

First, we define the time reversal of $\eta$ and $\eta_\phi$:
\begin{equation}\label{e.deftimereversal}
\xi(t,x)=\eta(-t,x), \quad\quad \xi_\phi(t,x)=\eta_\phi(-t,x)=\int_{\R^d} \phi(x-y)\xi(t,y)dy.
\end{equation}
Fix the inverse temperature $\beta>0$. For any $s\in\R$ and $x\in\R^d$, define $U(s,x;t,y)$ as the solution of
\begin{equation}\label{e.shephi}
\begin{aligned}
\partial_t U(s,x;t,y)&=\tfrac12\Delta_y U(s,x;t,y)+\beta\, U(s,x;t,y)\xi_\phi(t,y), \quad\quad t>s, y\in\R^d,\\
U(s,x;s,y)&=\delta(y-x).
\end{aligned}
\end{equation}
Here the product between $U$ and $\xi_\phi$ is interpreted in the It\^o-Walsh sense \cite{walsh1986introduction}, and we have included the spacetime white noise case of $\phi(\cdot)=\delta(\cdot)$ in $d=1$. Then the quenched endpoint density $q(T,x)$, defined in \eqref{e.endpoint} and \eqref{e.endpointwhite}, is also given by 
\begin{equation}\label{e.defq}
q(T,x)=\frac{\int_{\R^d}U(-T,x;0,y)q_0(y)dy}{\int_{\R^{2d}} U(-T,\tilde{x};0,y)q_0(y)dyd\tilde{x}}.
\end{equation}
To see this, consider the spatially correlated case with $\phi\in C_c^\infty(\R^d)$.  We only need to use Feynman-Kac formula \cite{bertini1995stochastic} to rewrite
\begin{equation}\label{e.fk}
\begin{aligned}
\int_{\R^d} U(-T,x;0,y)q_0(y)dy=&\E_{\mu_0}[\delta(w_T-x)e^{\beta\int_0^T\xi_\phi(-s,w_s)ds-\frac12\beta^2R(0)T}]\\
=&\E_{\mu_0}[\delta(w_T-x)e^{\beta H(T,w)-\frac12\beta^2R(0)T}].
\end{aligned}
\end{equation}
Here we recall that $w$ is a standard Brownian motion that is independent from $\xi$, and $\E_{\mu_0}$ denotes the expectation with respect to $w$, with the starting point
\[
w_0\sim \mu_0(dx)=q_0(x)dx.
\] 
For any $y\in\R^d$, let $\E_y$  denote the expectation with respect to the Brownian motion starting at $w_0=y$.  Then we can rewrite \eqref{e.fk} as
\[
\int_{\R^d} U(-T,x;0,y) q_0(y)dy=\int_{\R^d} q_0(y)\E_y[\delta(w_T-x)e^{\beta \int_0^T\xi_\phi(-s,w_s)ds-\frac12\beta^2R(0)T}]dy.
\]

For the case of $\phi(\cdot)=\delta(\cdot)$ in $d=1$, by the definition of the Wick-ordered exponential, we have 
\[
\begin{aligned}
&\E_{\mu_0}[\delta(w_T-x):\exp(\beta\textstyle\int_0^T \eta(s,w_s)ds):]\\
&=\E_{\mu_0}[\delta(w_T-x):\exp(\beta\textstyle\int_0^T \xi(-s,w_s)ds):]=\int_{\R^d}U(-T,x;0,y)q_0(y)dy.
\end{aligned}
\]

\section{Proof of Theorem~\ref{t.bbgky}}
\label{s.bbgky}

We make use of the Feynman-Kac representation in \eqref{e.endpoint} to study the case of spatially correlated noise, i.e., when the spatial covariance function $R(\cdot)\in C_c^\infty(\R^d)$. Through an approximation argument, we derive the corresponding result for the case of spacetime white noise.

\subsection{Colored noise environment: $R(\cdot)\in C_c^\infty(\R^d), d\geq1$}
We first introduce some notation. Fix $n\geq1,T>0$ and a $C_b^{1,2}$ function $f:[0,T]\times \R^{nd}\to \R$, we define 
\begin{equation}\label{e.defYZ}
\begin{aligned}
&M(t,w)=\exp(\beta H(t,w)-\tfrac12\beta^2R(0)t),\\
&Y_f(t)=f(t,w_t^1,\ldots,w_t^n)\textstyle\prod_{j=1}^nM(t,w^j),\ \ \ \text{ and}\\
&X_f(t)=\E_{\mu_0}[Y_f(t)],
\end{aligned}
\end{equation}
where $\{w^j\}_{j=1,\ldots,n}$ are independent copies of Brownian motions built on $(\Sigma,\mathcal{A},\Pb)$. Thus, we have 
\[
X_{\1}(t)=Z(t)^n,
\] where $Z(t)$ is the partition function defined in \eqref{e.endpoint} and  $\1$ stands for the constant function $\1(x)\equiv 1$. With the new notation, we define 
\[
\X_f(t)
	:=\la f(t,\cdot),\mu_t^{\otimes n}\ra
	=\la f(t,\cdot), q(t,\cdot)^{\otimes n}\ra=\frac{X_f(t)}{X_{\1}(t)}=\frac{X_f(t)}{Z(t)^n}.
\]

\begin{proof}[Proof of \eqref{e.eqP}]
In the following, the differential $d$ is the full stochastic differential with respect to both the Gaussian environment and the Brownian motions.

We first note that for each fixed $w$, 
\[
H(t,w)=\int_0^t \eta_\phi(s,w_s)ds
\] is a Brownian motion with variance $R(0)$, and for $w^i,w^j$, we have the bracket process
\[
\la H(\cdot,w^i),H(\cdot,w^j)\ra_t=\int_0^t R(w^i_s-w^j_s)ds.
\]This implies  
\begin{equation}\label{e.gbm}
\begin{aligned}
&dM(t,w)=\beta M(t,w)dH(t,w),\\
&d\la M(\cdot,w^i),M(\cdot,w^j)\ra_t=\beta^2 M(t,w^i)M(t,w^j) R(w^i_t-w^j_t)dt.
\end{aligned}
\end{equation}
We also know that   $H(t,w)$ is independent of $w$.

Now we apply It\^o's formula to $Y_f(t)$:
\[
\begin{aligned}
dY_f(t)
	=&d[f(t,w_t^1,\ldots,w_t^n)\textstyle\prod_{j=1}^n M(t,w^j)]\\
	=&f(t,w_t^1,\ldots,w_t^n)d[\textstyle\prod_{j=1}^n M(t,w^j)]+\textstyle\prod_{j=1}^n M(t,w^j)d[f(t,w_t^1,\ldots,w_t^n)]\\
	&+d\la f(\cdot,w_{\cdot}^1,\ldots,w_{\cdot}^n),\textstyle\prod_{j=1}^n M(\cdot,w^j)\ra_t.
\end{aligned}
\]
By \eqref{e.gbm}, we have 
\[
\begin{aligned}
d\textstyle\prod_{j=1}^n M(t,w^j)=&\beta \textstyle\sum_{k=1}^n\textstyle\prod_{j=1}^n M(t,w^j)dH(t,w^k)\\
&+\beta^2\textstyle\sum_{1\leq k<l\leq n} \textstyle\prod_{j=1}^n M(t,w^j)R(w_t^k-w_t^l)dt.
\end{aligned}
\]
We also have
\[
df(t,w_t^1,\ldots,w_t^n)=\textstyle\sum_{j=1}^n \nabla_j f(t,w_t^1,\ldots,w_t^j)dw_t^j+(\partial_t +\tfrac12\Delta )f(t,w_t^1,\ldots,w_t^n)dt,
\] 
where $\nabla_j$ denotes the gradient with respect to the $j-$th variable and $\Delta=\textstyle\sum_{j=1}^n \nabla_j\cdot\nabla_j$. 
Since $H(t,w)$ is independent of $w$, by \eqref{e.gbm}, we have
\[
\la f(w_{\cdot}^1,\ldots,w_{\cdot}^n),\textstyle\prod_{j=1}^n M(\cdot,w^j)\ra_t\equiv 0.
\]

Thus, we have
\[
\begin{aligned}
Y_f(T)=f(0,w_0^1,\ldots,w_0^n)&+\beta \sum_{k=1}^n\int_0^TY_f(t) dH(t,w^k)+\beta^2\sum_{1\leq k<l\leq n}\int_0^T Y_f(t) R(w_t^k-w_t^l) dt\\
&+\sum_{k=1}^n \int_0^T \prod_{j=1}^n M(t,w^j)\nabla_k f(t,w_t^1,\ldots,w_t^n)dw_t^k\\
&+\int_0^T \prod_{j=1}^nM(t,w^j)(\partial_t+\tfrac12\Delta) f(t,w_t^1,\ldots,w_t^n)dt.
\end{aligned}
\]
Taking expectation with respect to the Brownian motions, we have 
\begin{equation}\label{e.dZ}
\begin{aligned}
X_f(T)=\E_{\mu_0}[Y_f(T)]=&\E_{\mu_0}[f(0,w_0^1,\ldots,w_0^n)]+\beta\sum_{k=1}^n\E_{\mu_0}\big[\int_0^TY_f(t)dH(t,w^k)\big]\\
&+\beta^2\sum_{1\leq k<l\leq n}\int_0^T\E_{\mu_0}[ Y_f(t) R(w_t^k-w_t^l)] dt\\
&+\int_0^T \E_{\mu_0}\big[\prod_{j=1}^nM(t,w^j)(\partial_t+\tfrac12\Delta) f(t,w_t^1,\ldots,w_t^n)\big]dt.
\end{aligned}
\end{equation}
For the second term on the r.h.s., recall that $\eta_\phi(t,x)=\int_{\R^d}\phi(x-y)\eta(t,y)dy$ with $\eta$ a spacetime white noise, using \eqref{e.defHt}, we can rewrite it as 
\begin{equation}\label{e.dZma}
\begin{aligned}
\beta \textstyle\sum_{k=1}^n\E_{\mu_0}[\int_0^TY_f(t)dH(t,w^k)]=\beta\textstyle\sum_{k=1}^n \int_0^T \int_{\R^d} \E_{\mu_0}[Y_f(t) \phi(w^k_t-y)] \eta(t,y)dydt.
\end{aligned}
\end{equation}

Now we apply It\^o's formula to $\X_f(t)=X_f(t)Z(t)^{-n}$:
\begin{equation}\label{e.itoX}
d\X_f(t)=\frac{1}{Z(t)^n}dX_f(t)-\frac{nX_f(t)}{Z(t)^{n+1}}dZ(t)+\frac{n(n+1)X_f(t)}{2Z(t)^{n+2}} d\la Z,Z\ra_t-\frac{n}{Z(t)^{n+1}}d\la X_f,Z\ra_t.
\end{equation}
By \eqref{e.dZ}, the martingale component of $X_f$ is given by \eqref{e.dZma}. A simpler version of \eqref{e.dZ} also gives 
\[
Z(T)=1+\beta\int_0^T\int_{\R^d}\E_{\mu_0}[M(t,w)\phi(w_t-y)]\eta(t,y)dydt,
\] which implies
\begin{equation}\label{e.braket}
\begin{aligned}
d\la X_{f},Z\ra_t=&\beta^2\sum_{k=1}^n\big(\int_{\R^{d}}\E_{\mu_0}\big[f(t,w_t^1,\ldots,w_t^n) \prod_{j=1}^{n+1} M(t,w^j) \phi(w_t^k-y)\phi(w_t^{n+1}-y)\big] dy \big)dt\\
=&\beta^2\sum_{k=1}^n \E_{\mu_0}\big[f(t,w_t^1,\ldots,w_t^n)\prod_{j=1}^{n+1} M(t,w^j)  R(w_t^k-w_t^{n+1})\big] dt.
\end{aligned}
\end{equation}

Combining \eqref{e.dZ}, \eqref{e.dZma}, \eqref{e.itoX}, \eqref{e.braket}, we have
\begin{equation}\label{e.781}
\X_f(T)=\X_f(0)+\sum_{i=1}^4 I_i(T),
\end{equation}
with 
\[
\begin{aligned}
I_1(T)=\int_0^T\frac{1}{Z(t)^n}dX_f(t)=&\beta\sum_{k=1}^n\int_0^T\int_{\R^d}\frac{1}{Z(t)^n}\E_{\mu_0}[Y_f(t)\phi(w_t^k-y)] \eta(t,y)dydt\\
&+\beta^2\sum_{1\leq k<l\leq n}\int_0^T\frac{1}{Z(t)^n}\E_{\mu_0}[ Y_f(t) R(w_t^k-w_t^l)] dt\\
&+\int_0^T \frac{1}{Z(t)^n}\E_{\mu_0}[\prod_{j=1}^n M(t,w^j)(\partial_t+\tfrac12\Delta) f(t, w_t^1,\ldots,w_t^n)]dt,
\end{aligned}
\]
\[
\begin{aligned}
I_2(T)&=-n\int_0^T\frac{X_f(t)}{Z(t)^{n+1}}dZ(t)\\
&=-n\beta\int_0^T\int_{\R^d}\frac{X_f(t)}{Z(t)^{n+1}}\E_{\mu_0}[M(t,w)\phi(w_t-y)] \eta(t,y)dydt,
\end{aligned}
\]
\[
\begin{aligned}
I_3(T)&=\frac{n(n+1)}{2}\int_0^T \frac{X_f(t)}{Z(t)^{n+2}} d\la Z,Z\ra_t\\
&=\frac{\beta^2n(n+1)}{2}\int_0^T \frac{X_f(t)}{Z(t)^{n+2}}\E_{\mu_0}\big[\prod_{j=1}^2M(t,w^j)\cdot R(w^1_t-w^2_t)\big] dt,
\end{aligned}
\]
and 
\[
\begin{aligned}
I_4(T)&=-n\int_0^T \frac{1}{Z(t)^{n+1}}d\la X_f,Z\ra_t\\
&=-n\beta^2\sum_{k=1}^n\int_0^T\frac{1}{Z(t)^{n+1}}\E_{\mu_0}\big[f(t,w_t^1,\ldots,w_t^n)\prod_{j=1}^{n+1} M(t,w^j)  R(w_t^k-w_t^{n+1})\big]  dt.
\end{aligned}
\]
Taking the expectation with respect to $\eta$, we have
\begin{equation}
\begin{aligned}
\bfE[I_1(T)]&=\int_0^T \la \beta^2 f_{0,R}(t)+(\partial_t+\tfrac12\Delta) f(t),Q_n(t) \ra dt, \quad\quad \bfE[I_2(T)]=0,\\
\bfE[I_3(T)]&=\int_0^T \la \beta^2 f_{2,R}(t), Q_{n+2}(t)\ra dt, \quad\quad \bfE[I_4(T)]=\int_0^T \la \beta^2f_{1,R}(t),Q_{n+1}(t)\ra dt.
\end{aligned}
\end{equation}
It suffices to note that $\bfE[\X_f(T)]=\la f(T), Q_n(T)\ra$ and $\X_f(0)=\la f(0),q_0^{\otimes n}\ra$ to complete the proof of \eqref{e.eqP}.
\end{proof}
\gu{
\begin{remark}\label{r.781}
The equation \eqref{e.eqP} results from taking expectation with respect to the noise on both sides of  \eqref{e.781}, in which the martingale terms disappear. If we are interested in the size of the random fluctuations, it suffices to study the martingale terms in $I_1(T),I_2(T)$. Actually, an SPDE satisfied by $q(t,x)$ was derived in \cite[Proposition 3.2]{gk}.
 \end{remark}
}

\subsection{White noise environment: $R(\cdot)=\delta(\cdot), d=1$}

The proof in this case is through an approximation of the spacetime white noise by colored noise. Recall that we only consider $d=1$ in this case.

For each $\eps>0$, define 
\begin{equation}\label{e.defxieps}
\xi_\eps(t,x)=\int_{\R}\phi_\eps(x-y)\xi(t,y)dy, \quad \phi_\eps(x)=\tfrac{1}{\eps}\phi(\tfrac{x}{\eps})
\end{equation}
and the spatial covariance function 
\[
R_\eps(x)=\tfrac{1}{\eps}R(\tfrac{x}{\eps})=\int_{\R} \phi_\eps(x+y)\phi_\eps(y)dy.
\]
For any $s\in\R,x\in\R$, let $U_\eps$ be the solution to 
\begin{equation}\label{e.sheeps}
\begin{aligned}
\partial_t U_\eps(s,x;t,y)&=\tfrac12\Delta_y U_\eps(s,x;t,y)+\beta\,U_\eps(s,x;t,y)\xi_\eps(t,y), \quad\quad t>s,y\in\R,\\
U_\eps(s,x;s,y)&=\delta(y-x).
\end{aligned}
\end{equation}
In other words, $U_\eps$ solves \eqref{e.shephi} with $\xi_\phi$ replaced by $\xi_\eps$. Similarly, the quenched endpoint distribution of the polymer in the environment $\xi_\eps$, with the starting point distributed as $\mu_{0,\eps}(dx)=q_0(x)dx$, is defined as 
\begin{equation}
\mu_{T,\eps}(dx)=q_\eps(T,x)dx, \quad q_\eps(T,x)=\frac{\int_{\R}U_\eps(-T,x;0,y)q_0(y)dy}{\int_{\R^2} U_\eps(-T,\tilde{x};0,y)q_0(y)dyd\tilde{x}}.
\end{equation}

For any $n\geq 1,T\geq0$ and $\x_{1:n}=(x_1,\ldots,x_n)\in\R^n$, define 
\[
Q_{n,\eps}(T,\x_{1:n})=\bfE[q_\eps(T,x_1)\ldots q_\eps(T,x_n)].
\]
By \eqref{e.eqP} (for the case of $\phi \in C_c^\infty(\R^d)$), we have, for any $f\in C_b^{1,2}([0,T]\times \R^n)$, 
\begin{equation}\label{e.eqQeps}
\begin{aligned}
\la f(T),Q_{n,\eps}(T)\ra=\la f(0),q_0^{\otimes n}\ra&+\int_0^T \la (\partial_t+\tfrac12\Delta) f(t), Q_{n,\eps}(t)\ra dt\\
&+\beta^2\sum_{k=0}^2\int_0^T\la f_{k,\eps}(t),Q_{n+k,\eps}(t) \ra dt,
\end{aligned}
\end{equation}
with the shorthand notation $f_{k,\eps}:=f_{k,R_\eps}$. Recall that for any $R\in C_c^\infty(\R^d)$, the functions $f_{k,R}$ were defined in \eqref{e.defFR}, for $k=0,1,2$. 

To prove \eqref{e.eqP} for the case of $\phi(\cdot)=\delta(\cdot)$, we need the following two technical lemmas:
\begin{lemma}\label{l.bdrhoeps}
For any $p\geq 1, t>0,x\in\R$, $q_\eps(t,x)\to q(t,x)$ in $L^p(\Omega,\mathcal{F},\mathbf{P})$, as $\eps\to0$. 
The convergence is uniform for $x\in\R$ and $t$ in compact subsets of $(0,\infty)$. 
In addition, there exists $C=C(p,\beta,T)>0$ such that for all $\eps\in(0,1), t\in (0,T],x\in\R$,
\begin{equation}\label{e.bdrhoeps}
\bfE[|q_\eps(t,x)|^p]+\bfE[|q(t,x)|^p] \leq C (G_t\star q_0(x))^p .
\end{equation}
\end{lemma}

\begin{lemma}\label{l.conQ}
$Q_n$ is continuous on $(0,\infty)\times \R^n$. 
\end{lemma}

The proof of Lemmas~\ref{l.bdrhoeps} and \ref{l.conQ} is given in Appendix~\ref{s.she}.

\begin{corollary}\label{c.conQ}
There exist $C=C(\beta,T)>0$ such that for all $ t\in(0,T],(x_1,\ldots,x_n)\in\R^n$ and $\eps\in(0,1)$,
\[
Q_{n,\eps}(t,x_1,\ldots,x_n)+Q_n(t,x_1,\ldots,x_n)\leq C\prod_{j=1}^n G_t\star q_0(x_j).
\]
In addition, $Q_{n,\eps}(t,x_1,\ldots,x_n)\to Q_n(t,x_1,\ldots,x_n)$ as $\eps\to0$, and the convergence is uniform for $(x_1,\ldots,x_n)\in\R^n$ and $t$ in compact subsets of $(0,\infty)$.
\end{corollary}

Now we can finish the proof of Theorem~\ref{t.bbgky}.

\begin{proof}
We start from \eqref{e.eqQeps} and pass to the limit of $\eps\to0$ for each term. First, by Corollary~\ref{c.conQ}, we have
\[
\la f(T),Q_{n,\eps}(T)\ra\to \la f(T),Q_n(T)\ra, 
\]
and
\[
\int_0^T \la (\partial_t+ \tfrac12\Delta) f(t), Q_{n,\eps}(t)\ra dt\to \int_0^T \la (\partial_t+\tfrac12\Delta)f(t), Q_{n}(t)\ra dt.
\]
The rest of the $\eps-$dependent terms in \eqref{e.eqQeps} are treated in the same way, so we take $\int_0^T\la f_{0,\eps}(t),Q_{n,\eps}(t)\ra dt$ as an example: for any $t$,
\[
\begin{aligned}
\la f_{0,\eps}(t),Q_{n,\eps}(t)\ra=&\sum_{1\leq i<j\leq n}\int_{\R^n} f(t,\x_{1:n})R_\eps(x_i-x_j)Q_{n,\eps}(t,\x_{1:n}) d\x_{1:n}.
\end{aligned}
\]
It suffices to consider fixed $i,j$ from the summation. By the change of variable $x_i\mapsto x_i, x_j\mapsto x_i-\eps x_j$, the integral equals to 
\[
\begin{aligned}
\int_{\R^n}& f(t,x_1,\ldots,x_i,\ldots,x_i-\eps x_j,\ldots x_n)R(x_j)\\
&\times Q_{n,\eps}(t,x_1,\ldots,x_i,\ldots,x_i-\eps x_j,\ldots x_n) d\x_{1:n}.
\end{aligned}
\]
By Corollary~\ref{c.conQ}, we have
\[
Q_{n,\eps}(t,x_1,\ldots,x_i,\ldots,x_i-\eps x_j,\ldots x_n) \leq Ct^{-\frac{1}{2}}\prod_{\ell:\,\ell \neq j} G_t\star q_0(x_\ell),
\]
where we also used the elementary estimate $G_t\star q_0(x_i-\eps x_j) \leq Ct^{-\frac12}$, which clearly holds for the two cases of $q_0$ we considered in the paper: $q_0\in C_c(\R)$ or $q_0(x)=\delta(x)$. For fixed $t\in(0,T)$ and $(x_1,\ldots,x_n)\in\R^n$, by the continuity of $f$, Lemma~\ref{l.conQ} and Corollary~\ref{c.conQ}, we obtain
\[
\begin{aligned}
&f(t,x_1,\ldots,x_i,\ldots,x_i-\eps x_j,\ldots x_n)Q_{n,\eps}(t,x_1,\ldots,x_i,\ldots,x_i-\eps x_j,\ldots x_n)\\
&\to f(t,x_1,\ldots,x_i,\ldots,x_i,\ldots x_n)Q_{n}(t,x_1,\ldots,x_i,\ldots,x_i,\ldots x_n), \quad \mbox{ as } \eps\to0.
\end{aligned}
\]
Note that $\int R=1$, we can apply dominated convergence theorem to conclude that 
\[
\int_0^T\la f_{0,\eps}(t),Q_{n,\eps}(t)\ra dt\to \int_0^T\la f_{0,\delta(\cdot)}(t),Q_{n}(t)\ra dt, \quad\mbox{ as }\eps\to0.
\]
Here we recall from \eqref{e.defFR} that
\[
f_{0,\delta(\cdot)}(t,\x_{1:n})=f(t,\x_{1:n})\sum_{1\leq i<j\leq n}\delta(x_i-x_j).
\]
The proof is complete.
\end{proof}

\subsection{Proof of Corollary~\ref{c.generator}}
 
The proof of the cases $R(\cdot)\in C_c^\infty(\R^d)$ and $R(\cdot)=\delta(\cdot)$ are similar, and we only deal with the latter. Fix $n=1$, $q_0\in C_c(\R^d)$ and $f\in C_b^2(\R^{nd})$, by Theorem~\ref{t.bbgky}, we have 
\[
\la f, Q_1(T)\ra=\la f, q_0\ra+\int_0^T \la \tfrac12\Delta f, Q_1(t)\ra dt+\beta^2\sum_{k=0}^2\int_0^T\la f_{k,\delta(\cdot)},Q_{1+k}(t) \ra dt.
\]
Recall that $\rmF_f(\mu_T)=\la f, \mu_T\ra$, so we have 
\[
\begin{aligned}
\frac{\bfE[\rmF_f(\mu_T)]-\rmF_f(\mu_0)}{T}=& \frac{\la f, Q_1(T)\ra-\la f, q_0\ra}{T}\\
=&\frac{1}{T}\int_0^T \la \tfrac12\Delta f, Q_1(t)\ra dt+\beta^2\sum_{k=0}^2 \frac{1}{T}\int_0^T\la f_{k,\delta(\cdot)},Q_{1+k}(t) \ra dt.
\end{aligned}
\]
By definition $f_{0,\delta(\cdot)}=0$ when $n=1$. For $k=1$, we have 
\[
\begin{aligned}
\frac{1}{T}\int_0^T\la f_{k,\delta(\cdot)},Q_{1+k}(t) \ra dt=&-\frac{1}{T}\int_0^T \int_{\R^2} f(x_1)\delta(x_1-x_2) Q_2(t,x_1,x_2)dx_1dx_2dt\\
=&-\frac{1}{T}\int_0^T \int_\R f(x_1)Q_2(t,x_1,x_1) dx_1dt\\
=&-\int_0^1\int_{\R} f(x_1)Q_2(Tt,x_1,x_1)dx_1dt.
\end{aligned}
\]
Similarly, when $k=2$, we have 
\[
\begin{aligned}
\frac{1}{T}\int_0^T\la f_{k,\delta(\cdot)},Q_{1+k}(t) \ra dt=&\frac{1}{T}\int_0^T \int_{\R^3} f(x_1)\delta(x_2-x_3)Q_3(t,x_1,x_2,x_3) dx_1dx_2dx_3dt\\
=&\frac{1}{T}\int_0^T \int_{\R^2} f(x_1)Q_3(t,x_1,x_2,x_2) dx_1dx_2dt\\
=&\int_0^1 \int_{\R^2} f(x_1)Q_3(Tt,x_1,x_2,x_2)dx_1dx_2dt.
\end{aligned}
\]
By applying Corollary~\ref{c.conQ} and Lemma~\ref{l.tzero} below, we have
\[
\begin{aligned}
\frac{\bfE[\rmF_f(\mu_T)]-\rmF_f(\mu_0)}{T}\to \la \tfrac12\Delta f,q_0\ra&-\beta^2\int_{\R}f(x_1)q_0(x_1)^2dx_1\\
&+\beta^2\int_{\R^2} f(x_1)q_0(x_1)q_0(x_2)^2dx_1dx_2  
\end{aligned}
\]
as $T\to0$. The r.h.s.\ equals to 
\[
\la \tfrac12\Delta f,q_0\ra+ \beta^2\la f_{1,\delta(\cdot)},q_0^{\otimes 2}\ra+\beta^2\la f_{2,\delta(\cdot)},q_0^{\otimes 3}\ra=\la \tfrac12\Delta f,q_0\ra+\beta^2\la f, \mathcal{T} q_0\ra,
\] 
which completes the proof of \eqref{e.generatorcolor}.

\begin{lemma}\label{l.tzero}
Assume $q_0\in C_c(\R^d)$. For any $n\geq1$ and $(x_1,\ldots,x_n)\in\R^{nd}$, as $t\to0$,
\[
\begin{aligned}
Q_n(t,x_1,\ldots,x_n)&\to \prod_{j=1}^n q_0(x_j).
\end{aligned}
\]
\end{lemma}

The proof of Lemma~\ref{l.tzero} is given in Appendix~\ref{s.she}.

\section{Quantitative central limit theorem: proofs of Theorems~\ref{t.qclt} and \ref{t.msd}}
\label{s.qclt}

In this section, we consider the high dimensions $d\geq3$ and a  high temperature regime with $\beta\ll1$. The goal is to prove Theorems~\ref{t.qclt} and \ref{t.msd}. With a change of variable, \eqref{e.wdbd} and \eqref{e.msdbd} are equivalent with, for any $h\in \mathrm{Lip}(1)$,  
\begin{equation}\label{e.201}
\left|\int_{\R^d} h(\eps x) Q_1(\tfrac{1}{\eps^2},x)dx-\int_{\R^d} h(x)G_1(x)dx \right| \leq C\Big( \eps|\log \eps|\1_{d=3}+\eps\1_{d\geq 4} \Big), 
\end{equation}
and
\begin{equation}\label{e.202}
\left| \int_{\R^d} |\eps x|^2 Q_1(\tfrac{1}{\eps^2},x)dx - d\right| \leq C \Big(\eps\1_{d=3}+\eps^2|\log \eps|\1_{d=4}+\eps^2\1_{d\geq5}\Big).
\end{equation}
To unify the notation, we view~\eqref{e.202} as a special case of~\eqref{e.201} with $h(x) = |x|^2$ even though this choice of $h$ is not an element of $\mathrm{Lip}(1)$. 
Note that, although the function $h$ here is not necessarily bounded, it grows at most polynomially at infinity. Thus, by Corollary~\ref{c.conQ}, it is easy to see that the two integrals in \eqref{e.201} are both well-defined.  Our proof below is based on selecting appropriate test functions in the equation satisfied by $Q_1$ in order to quantify the cancellation between the $Q_2$ and $Q_3$ terms.

Throughout the section, we assume that $q_0(\cdot)=\delta(\cdot)$; that is, the starting point of the polymer path is at the origin. Recall that in high dimensions we assumed the random environment is smooth in the spatial variable and $R(\cdot)\in C_c^\infty(\R^d)$ is the spatial covariance function. 

\subsection{Error form}
The first step is to derive an exact error expression in \eqref{e.201} using the hierarchical PDE system. We define an auxiliary test function as follows: for any $\eps>0$ and a function $h$, let $f_\eps(t,x)$ be the solution to the backward heat equation
\begin{equation}\label{e.backheat}
\begin{aligned}
&\partial_t f_\eps(t,x)+\tfrac12\Delta f_\eps(t,x)=0,\quad\quad  t<\tfrac{1}{\eps^2},x\in\R^d,\\
&f_\eps(\tfrac{1}{\eps^2},x)=h(\eps x).
\end{aligned}
\end{equation}
Then we have
\begin{lemma}\label{l.errde}
For a continuous function $h$ with at most polynomial growth at infinity, we have 
\begin{equation}\label{e.203}
\begin{aligned}
\mathcal{E}_\eps(h)&:=\int_{\R^d} h(\eps x)  Q_1(\tfrac{1}{\eps^2},x) dx- \int_{\R^d} h(x) G_1(x)dx\\
&=\beta^2\int_0^{\eps^{-2}}\int_{\R^{3d}}[f_\eps(t,x)-f_\eps(t,y)]R(y-z)Q_3(t,x,y,z)dxdydzdt.
\end{aligned}
\end{equation}
\end{lemma}

\begin{proof}
We first assume in addition that $h\in C_b^2(\R^d)$. As $f_\eps$ solves the backward heat equation, it holds that $f_\eps\in C_b^{1,2}([0,\eps^{-2}]\times\R^d)$. In the hierarchical PDE system \eqref{e.eqP}, we take $n=1,T=\eps^{-2}$, and the test function to be $f_\eps$ to obtain
\begin{equation}\label{e.131}
\begin{split}
\int_{\R^d} f_\eps(\tfrac{1}{\eps^2},x)&Q_1(\tfrac{1}{\eps^2},x)dx=\int_{\R^d} f_\eps(0,x)Q_1(0,x)dx\\
&-\beta^2\int_0^{\eps^{-2}}\int_{\R^{2d}}f_\eps(t,x)R(x-y)Q_2(t,x,y)dxdydt\\
&+\beta^2\int_0^{\eps^{-2}}\int_{\R^{3d}}f_\eps(t,x)R(y-z)Q_3(t,x,y,z)dxdydzdt.
\end{split}
\end{equation}
As $Q_1(0,x)=q_0(x)=\delta(x)$, the first term on the r.h.s.\ of \eqref{e.131} is
\begin{equation}\label{e.204}
\begin{aligned}
\int_{\R^d}f_\eps(0,x)Q_1(0,x)dx=f_\eps(0,0)=\int_{\R^{d}}h(\eps x)G_{\eps^{-2}}(x)dx=\int_{\R^{2d}} h(x)G_1(x)dx,
\end{aligned}
\end{equation}
where the last step is through a change of variable and using the scaling property of the heat kernel. 

By definition, $Q_n(t,x_1,\ldots,x_n)$ is symmetric in the $x-$variables and 
\[
	Q_2(t,x,y)=\int_{\R^d} Q_3(t,x,y,z)dz.
\]
Thus, \eqref{e.131} can be rewritten as 
\[
	\begin{aligned}
		\mathcal{E}_\eps(h)
		=\beta^2\int_0^{\eps^{-2}}\int_{\R^{3d}}[f_\eps(t,x)-f_\eps(t,y)]R(y-z)Q_3(t,x,y,z)dxdydzdt.
	\end{aligned}
\]
where we also used the fact that $R(\cdot)$ is even. Through an approximation and the bound on $Q_n$ given in Corollary~\ref{c.conQ}, the above identity extends to the case of $h$ having at most polynomial growth at infinity, which completes the proof.
\end{proof}

\begin{remark}
For the case of $q_0(\cdot)\in C_c(\R^d)$, a similar error decomposition as \eqref{e.203} can be derived. The only change to make in the proof is in \eqref{e.204}, where an extra error term comes out of the weak convergence of $\tfrac{1}{\eps^d} q_0(\tfrac{x}{\eps})\to \delta(x)$.
\end{remark}

\subsection{Estimating $\mathcal{E}_\eps(h)$}

The proof of Theorems~\ref{t.qclt} and \ref{t.msd} reduces to the estimate of $\mathcal{E}_\eps(h)$ for $h\in \mathrm{Lip}(1)$ and $h(x)=|x|^2$ respectively. By using a probabilistic representation, the following bounds on $Q_n$ hold in the high temperature regime in $d\geq3$:
\begin{lemma}\label{l.bdQn}
For any $d\geq3$ and $n\geq1$, there exists constants $\beta_0(d,n,R)>0$ and $C(d,n,R,\beta)$ such that if $\beta<\beta_0(d,n,R)$, we have 
\[
Q_n(t,x_1,\ldots,x_n)\leq C \prod_{j=1}^n G_t(x_j), \quad\quad \mbox{ for all } t>0,x_1,\ldots,x_n\in\R^d.
\]
\end{lemma}
The proof of Lemma~\ref{l.bdQn} is in Appendix~\ref{s.lem}.

Before undertaking the proofs of \Cref{t.qclt,t.msd}, we provide a heuristic argument that shows how the diffusive behavior of the polymer endpoint behavior follows from the convergence of the error $\mathcal{E}_\eps(h)$ to zero.  
Recall that
\[
\mathcal{E}_\eps(h)=\beta^2\int_0^{\eps^{-2}}\int_{\R^{3d}}[f_\eps(t,x)-f_\eps(t,y)]R(y-z)Q_3(t,x,y,z)dxdydzdt.
\]
Consider the simple case of $h\in C_b(\R^d)$, so $|f_\eps|\leq \sup_x|h(x)|$. The key point here is that, with the assumption of $\beta\ll1$ and $d\geq3$,
\[
\int_{\R^{3d}}R(y-z)Q_3(t,x,y,z)dxdydz=\int_{\R^{2d}}R(y-z) Q_2(t,y,z)dydz \leq C t^{-d/2}
\]
and, hence, is integrable for $t\in [1,\infty)$.  This is ultimately related to the fast decay of the heat kernel in high dimensions $d\geq3$, with the smallness of $\beta$ ensuring that the effect of the random environment is ``summable'' in the limit (hidden in the proof of Lemma~\ref{l.bdQn}). Therefore, the main contribution to $\mathcal{E}_\eps(h)$ actually comes from the time integration in a \emph{microscopically} large domain $[0,M]$ for $1 \ll M \ll \eps^{-2}$; that is
\[
\mathcal{E}_\eps(h)\approx \beta^2\int_0^{M}\int_{\R^{3d}}[f_\eps(t,x)-f_\eps(t,y)]R(y-z)Q_3(t,x,y,z)dxdydzdt. 
\]
On the other hand, when $t \leq M \ll \eps^{-2}$, the $f$ terms cancel for the following reason: for any fixed $(t,x)$, it is straightforward to check that
\[
f_\eps(t,x)=\int_{\R^d} h(\eps z) G_{\eps^{-2}-t}(x-z)dz=\int_{\R^d} h(z) G_{1-\eps^2t}(\eps x-z)dz\to \int_{\R^d} h(z)G_1(z)dz,
\]
which is independent of $x$. Thus, by the dominated convergence theorem, we obtain that $\mathcal{E}_\eps(h)\to0$. The proofs of Theorems~\ref{t.qclt} and \ref{t.msd} then rely on quantifying this argument, which we do now.

In the proofs below, $C$ is a constant independent of $\eps$ which may change from line to line.


\begin{proof}[Proof of Theorem~\ref{t.qclt}]
Fix $h\in \mathrm{Lip}(1)$, we have 
\[
|f_\eps(t,x)-f_\eps(t,y)|\leq  \int_{\R^d}G_{\eps^{-2}-t}(z) |h(\eps(x-z))-h(\eps(y-z))| dz \leq \eps \,|x-y|, 
\]
which implies
\begin{equation}\label{e.bdEeps}
\begin{aligned}
|\mathcal{E}_\eps(h)| \leq\, &\beta^2\eps\int_0^{\eps^{-2}}\int_{\R^{3d}} |x-y|R(y-z)Q_3(t,x,y,z)dxdydzdt\\
\leq\, & C\eps\int_0^{\eps^{-2}}\int_{\R^{3d}} |x-y| R(y-z) G_t(x)G_t(y)G_t(z)dxdydzdt.
\end{aligned}
\end{equation}
While the above integral can be estimated directly (as in the proof of Theorem~\ref{t.msd} below), we present a simple probabilistic argument here. Let $B^1,B^2,B^3$ be three independent Brownian motions, then the integral can be written as
\[
\int_{\R^{3d}} |x-y| R(y-z) G_t(x)G_t(y)G_t(z)dxdydz=\E[|B^2_t-B^1_t|\cdot R(B^2_t-B^3_t)].
\]
Since $B^2-B^1$ and $B^2-B^3$ are correlated Brownian motions with variance $2$ and covariance $1$, we can rewrite the expectation as
\[
\E[|B^2_t-B^1_t|\cdot R(B^2_t-B^3_t)]
	=\E\Big[\Big|\sqrt{\tfrac{1}{2}}W^1_t-\sqrt{\tfrac32}W^2_t\Big| \cdot R(\sqrt{2}W^1_t)\Big],
\]
with $W^1,W^2$ independent Brownian motions. Since $R\in C_c^\infty$, we estimate the above expectation by 
\[
	\E\Big[\Big|\sqrt{\tfrac{1}{2}}W^1_t-\sqrt{\tfrac32}W^2_t\Big| \cdot R(\sqrt{2}W^1_t)\Big]
		\leq C  \E\Big[\Big|\sqrt{\tfrac12}W^1_t-\sqrt{\tfrac32}W^2_t\Big|\cdot \1_{\{|W^1_t|\leq M\}}\Big]
\]
for some constant $M>0$. For $t\leq 1$, we have the obvious bound 
\[
\E\Big[\Big|\sqrt{\tfrac12}W^1_t-\sqrt{\tfrac32}W^2_t\Big|\cdot\1_{\{|W^1_t|\leq M\}}\Big]\leq C.
\]
For $t>1$, by first averaging $W^2$, we have 
\[
\E\Big[\Big|\sqrt{\tfrac{1}{2}}W^1_t-\sqrt{\tfrac{3}{2}}W^2_t\Big|\cdot \1_{\{|W^1_t|\leq M\}}\Big] \leq C\sqrt{t}\,\Pb[|W^1_t|\leq M] \leq C t^{-\frac{d-1}{2}}.
\]
Combining the two cases and plugging into \eqref{e.bdEeps},  we derive
\[
|\mathcal{E}_\eps(h)|
	\leq C\eps\, \Big(1+\int_1^{\eps^{-2}}t^{-\frac{d-1}{2}}dt\Big)
	\leq C\Big(\eps\, |\log \eps|\1_{d=3} +\eps\1_{d\geq4}\Big).
\] 
The proof is complete.
\end{proof}

\begin{proof}[Proof of Theorem~\ref{t.msd}]
Let $h(x)=|x|^2$, then 
\[
f_\eps(t,x)=\int_{\R^d} G_{\eps^{-2}-t}(y)|\eps (x-y)|^2 dy=\eps^2|x|^2+(1-\eps^2 t)d.
\]
Applying Lemmas~\ref{l.errde} and \ref{l.bdQn}, we have 
\begin{equation}\label{e.205}
\begin{aligned}
|\mathcal{E}_\eps(h)|=\left|\beta^2\eps^2 \int_0^{\eps^{-2}}\int_{\R^{3d}}(|x|^2-|y|^2)R(y-z)Q_3(t,x,y,z)dxdydzdt\right|\\
\leq C \eps^2 \int_0^{\eps^{-2}}\int_{\R^{3d}} ||x|^2-|y|^2| R(y-z) G_t(x)G_t(y)G_t(z)dxdydzdt.
\end{aligned}
\end{equation}
To estimate the above integral, as $R$ is compactly supported (suppose its support has a radius $M>0$), we have 
\[
\int_{\R^d} R(y-z) G_t(z) dz \leq C \int_{\R^d} \1_{\{|y-z| \leq M\}} G_t(z)dz \leq C\Big(\1_{t\leq 1}+\1_{t>1} \sup_{z:|z-y|\leq M} G_t(z)\Big).
\]
For the case of $t>1$, by considering $|y|\leq 2M$ and $|y|>2M$ separately, we derive
\[
\sup_{z:|z-y|\leq M} G_t(z) \leq C G_{c_1t}(y),  \quad \mbox{ for some } c_1>0.
\]
From \eqref{e.205} and the above estimate, the following bound holds
\[
\begin{aligned}
|\mathcal{E}_\eps(h)| \leq& C\eps^2 \int_0^1 \int_{\R^{2d}}||x|^2-|y|^2|  G_t(x)G_t(y) dxdydt\\
&+C\eps^2\int_1^{\eps^{-2}}\int_{\R^{2d}}||x|^2-|y|^2| G_t(x)G_t(y)G_{c_1t}(y)dxdydt=A_1+A_2.
\end{aligned}
\]
We first get $A_1\leq C\eps^2$. For $A_2$, by the fact that 
\[
G_t(y)G_{c_1t}(y) \leq C t^{-d/2}G_{c_2t}(y), \quad \mbox{ for some } c_2>0,
\]
we have 
\[
\begin{aligned}
A_2&\leq C\eps^2\int_1^{\eps^{-2}}\int_{\R^{2d}}t^{-d/2}||x|^2-|y|^2| G_{t}(x)G_{c_2t}(y)dxdydt \\
&\leq C \eps^2\int_1^{\eps^{-2}}t^{1-d/2}dt\leq C\Big( \eps\1_{d=3}+\eps^2|\log \eps|\1_{d=4}+\eps^2\1_{d\geq5}\Big).
\end{aligned}
\]
This completes the proof.
\end{proof}

\section{Growth of moments: proof of Theorem~\ref{t.23}}
\label{s.23}

In the interest of the simplest presentation, we remove the parameter $\beta$ by scaling.  Indeed, let $\overline g(t,x) = \beta^{-2} g(t \beta^{-4}, x \beta^{-2})$, and observe that
\[
	\overline g_t
		= \overline g_{xx} + \overline g (\|\overline g\|^2 - \overline g).
\]

Hence, for the remainder of the section, we set $\beta = 1$; that is, we are interested in 
\begin{equation}\label{e.maineq1}
\begin{aligned}
\partial_t g(t,x)
&=\tfrac12\Delta g(t,x)+ \|g(t,\cdot)\|^2 g(t,x)- g(t,x)^2, \quad\quad t>0, x\in\R,\\
g(0,x)&=q_0(x).
\end{aligned}
\end{equation}
Here we abused notation by reverting to $g$ as opposed to using $\overline g$.  Undoing this simple scaling reveals the dependence on $\beta$ of our results.

In order to control the moments of $g$, it is necessary to understand the asymptotic behavior of $\|g\|^2$ as $t\to\infty$.  By interpolation and the fact that $g$ is a probability density (noted below), it is enough to control the maximum of $g$.  In the following section, we state the main estimate on the decay of the maximum of $g$.  After, we show how to use this to obtain upper and lower bounds on the moments of $g$ by constructing sharp sub and supersolutions of $g$.  In \Cref{s.max_decay}, we show how to obtain the correct asymptotics on the maximum of $g$.

\subsection{Statement of the main inequality and its application to the moments of $g$}

In order to streamline the argument, we define a few quantities that play key roles in the proof.  For any $t \geq 0$, let
\be\label{e.M_E_D}
	M(t) = \max_{x\in\R} g(t,x),
	\quad
	E(t) = \int_{\R} g(t,x)^2 dx,
	\quad \text{ and }\quad
	D(t) = \int_{\R} |g_x(t,x)|^2 dx.
\ee

The key inequality that we require is stated in the following proposition, proved in \Cref{s.max_decay}.
\begin{proposition}\label{p.max_decay}
	There is a universal constant $C_0$, independent of the initial data, such that
	\[
		M(t)
			\leq \frac{C_0}{t^{2/3}}
			\qquad\text{ for all } t >0.
	\]
\end{proposition}

Two more useful facts are the following.  Integrating~\eqref{e.maineq1}, we see that
\[
	\frac{d}{dt} \int_{\R} g(t,x) dx = E(t) \left(\int_{\R} g(t,x)dx - 1\right)
\]
Since $\int g(0,x) dx =\int q_0(x)dx=1$, by assumption, a simple ODE argument yields, for any $t\geq 0$,
\be\label{e.mass_one}
	\int_{\R} g(t,x) dx = \int_{\R} q_0(x) dx = 1.
\ee
Thus $g(t,\cdot)$ is a probability density (note that $g\geq0$ by comparison principle and the fact that $q_0\geq0$).  This is unsurprising given the derivation of the model~\eqref{e.maineq1}; however, it is crucial in our analysis.  Indeed, we immediately deduce the following useful inequality:
\be\label{e.E_less_M}
	E(t)
		\leq M(t) \int_{\R} g(t,x) dx
		= M(t).
\ee

We now show how to conclude \Cref{t.23} assuming \Cref{p.max_decay}.  We begin with the upper bound.

\begin{proof}[Proof of the upper bound in \Cref{t.23}]
	The first step is to replace the $\|g\|^2 = E$ term in~\eqref{e.maineq1}.  From \Cref{p.max_decay} and~\eqref{e.E_less_M}, we see that
	\[
		E(t) \leq \frac{C_0}{t^{2/3}}.
	\]
	
	This, along with~\eqref{e.maineq1}, implies that
	\[
		\partial_t g - \frac{1}{2} \Delta g - \frac{C_0}{t^{2/3}} g \leq 0.
	\]
	The comparison principle implies that $g\leq \overline g$, where $\overline g$ solves
	\be\label{e.supersoln}
		\begin{cases}
			\partial_t \overline g
				- \frac{1}{2} \Delta \overline g - \frac{C_0}{t^{2/3}} \overline g = 0
					\qquad &\text{ in } (0,\infty) \times \R,\\
			\overline g = q_0
						&\text{ on } \{0\}\times \R.
		\end{cases}
	\ee
	
	The second step is to obtain a bound on $\overline g$, and, hence, on $g$, for large $x$.  The first thing to notice is that
	\[
		\overline h(t,x) = \exp\left\{- \int_0^t \frac{C_0}{s^{2/3}} ds\right\} \overline g(t,x)
	\]
	solves the heat equation, $\partial_t\overline h = \frac{1}{2} \Delta \overline h$.  It follows that
	\[
		\begin{split}
			\overline g(t,x)
				&= \exp\left\{\int_0^t \frac{C_0}{s^{2/3}} ds\right\}
					\overline h(t,x)\\
				&= \exp\left\{\int_0^t \frac{C_0}{s^{2/3}} ds\right\} \int_{\R} \frac{1}{\sqrt{2\pi t}} e^{- \frac{y^2}{2t}}  q_0(x-y) dy.
		\end{split}
	\]
	By assumption, $q_0$ is compactly supported.  A straightforward estimate of the convolution, as well as a simple evaluation of the time integral, yields, for any $t \geq 1$ and any $x$,
	\[
		g(t,x)
			\leq \overline g(t,x)
			\leq \frac{C}{\sqrt t} e^{3C_0 t^{1/3} - \frac{x^2}{2t}}
	\]
	for some positive constant $C$ depending only on the initial data.
	
	We now conclude the bound on the moments of $g$.  Pairing the above arguments with~\eqref{p.max_decay}, we have established that, for all $t\geq 1$,
	\[
		g(t,x)
			\leq C \min\left\{\frac{1}{t^{2/3}},  \frac{1}{t^{1/2}} \exp\left\{ 3C_0 t^{1/3} - \frac{x^2}{2t}\right\}\right\}.
	\]
	We now use this to conclude the proof.  Indeed, for any $p\geq 1$, we find
	\[\begin{split}
		\int_{\R} |x|^p g(t,x) dx
			&\leq \int_{|x| \leq 6 C_0 t^{2/3}} |x|^p \frac{C}{t^{2/3}} dx
				+ \int_{|x| > 6 C_0 t^{2/3}} |x|^p\frac{C}{t^{1/2}} e^{3 C_0 t^{1/3} - \frac{x^2}{2t}} dx\\
			&\leq C t^\frac{2p}{3}
				+ Ct^{p/2} \int_{|y| > 6 C_0 t^{1/6}} |y|^p e^{3 C_0 t^{1/3} - \frac{y^2}{2}} dy\\
			&\leq C t^\frac{2p}{3}
				+ C t^{p/2} t^{(p-1)/6} e^{ - (2(3C_0)^2 - (3C_0)) t^{1/3}}.
	\end{split}\]
	where $C$ is a constant depending only on $q_0$ and $p$ that changes line-by-line.  The second term clearly tends to zero.  This completes the proof.
\end{proof}

It is now possible to deduce the lower bound using the upper bound \gu{given in Proposition~\ref{p.max_decay}} and~\eqref{e.mass_one}.  We require one lemma.
\begin{lemma}\label{l.minimizer}
Assume that $\lambda > 0$, $d\geq 1$, and $w: [0,\infty) \to \R$ is an increasing function.   Then
\[
	\min_{
			\substack{
				\int g(x) dx = 1,\\ 0 \leq g \leq \lambda
			}
		}
		\int_{\R^d} w(|x|) g(x)dx
		= \lambda \int_{B_{(\lambda \omega_d)^{-1/d}}} w(|x|) dx.
\]
where $B_r$ denotes the ball centered at the origin with radius $r$ and $\omega_d$ is the volume of the unit ball in $\R^d$.
\end{lemma}

This lemma is elementary and follows from the fact that the minimizer is clearly $\lambda \1_{B_{(\lambda \omega_d)^{-1/d}}}(x)$.  Hence, we omit its proof.  We now conclude the proof of \Cref{t.23}.

\begin{proof}[Proof of the lower bound in \Cref{t.23}]

From \Cref{p.max_decay}, we know that $g(t,x)\leq \frac{C_0}{(1+t)^{2/3}}=:\lambda$ 
for all $t$.  Hence
\[
	\int_{\R} |x|^p g(t,x)dx
		\geq \min_{\substack{\int \overline g(x) dx = 1,\\ 0 \leq \overline g \leq \lambda}} 
			\int_\R |x|^p \overline g(x) dx.
\]
Applying \Cref{l.minimizer}, we have 
\[
	\min_{\substack{\int \overline g(x) dx = 1,\\ 0 \leq \overline g \leq \lambda}}
			\int_\R |x|^p \overline g(x) dx
		\geq \lambda \int_{-1/2\lambda}^{1/2\lambda} |x|^p dx
		= \frac{2^{-(p+1)}}{p+1} \lambda^{-p}
		=  \frac{2^{-(p+1)}}{p+1} \left(\frac{(1+t)^{2/3}}{C_0}\right)^p,
\]
which concludes the proof.
\end{proof}

\subsection{Decay of the maximum of $g$}\label{s.max_decay}

Classical techniques for decay of parabolic equations are often based on Nash's inequality, which relates the $L^2$ norm of the gradient of $g$ with the $L^2$ norm of $g$.  Such an estimate necessarily gives decay like $O(t^{-1/2})$ in $d=1$, which is slower than the rate of decay we prove below.  Hence, such a strategy is not useful here.

In other words, the Laplacian term (and the related Dirichlet energy $D$) is not sufficient to obtain decay like $O(t^{-2/3})$.  The only other term in the equation is $g(E-g)$, and, hence, our proof must be based on this term.  The key observation is that near the maximum of $g$, we expect $g(E - g) \approx M(E-M) < 0$.    As such, we require an estimate that quantifies how negative this term is. 

In fact, our argument is more subtle than this.  We use the decay induced by both terms $-D$ and $-M(M-E)$.  Indeed, if $M-E$ is large, then the nonlinear term $-M(M-E)$ is a large negative number.  On the other hand, if $M \approx E$, it must be that $g$ ``flattens'' quickly after reaching the maximum, making $D$ large (recall that $g$ is a probability measure so if $M \approx E$ then $g$ is near the optimal case in H\"older's inequality, which, in turn, implies that $g$ is nearly an indicator function).  In both cases, we get a large decay term. 
The key estimate quantifying this heuristic is the following, which is proved at the end of the section.
\begin{lemma}\label{l.dissipation}
There is a universal constant $C_1>0$ such that
\[
	M - E
		\geq \frac{M^4}{C_1 D}.
\]
\end{lemma}

Before beginning the proof of \Cref{p.max_decay}, we collect two more inequalities.  The first is that, for any $0 < t_1 < t_2$,
\be\label{e:mp}
	M(t_2)
		\leq M(t_1) +  \int_{t_1}^{t_2} (E(s) M(s) - M(s)^2) ds.
\ee
Informally, this can be seen by noting that, at a maximum, $\Delta g \leq 0$, so that~\eqref{e.maineq1} reads $\dot M \leq EM - M^2$, where we used the physics notation $\cdot$ to denote the time derivative. This differential inequality has to be interpreted in the suitable weak sense, but this purely technical issue is standard in parabolic theory and, hence, we omit the details.

The second inequality is, for all $t_1 < t_2$,
\be\label{e:e_i2}
	E(t_2)
		\leq E(t_1) - \int_{t_1}^{t_2} D(s) ds.
\ee
In order to see this, simply multiply~\eqref{e.maineq1} by $g$ and integrate in $x$ in order to obtain
\[
	\frac{1}{2} \dot E
		+ \frac{1}{2} D
		\leq E^2 - \int g^3 .
\]
Since
\[
	E(t)
		= \int_{\R} g^{3/2}(t,x) g^{1/2}(t,x) dx
		\leq \left( \int_\R g^3(t,x) dx \right)^{1/2} \left( \int_\R g(t,x) dx \right)^{1/2},
\]
then $\frac{1}{2} \dot E + \frac{1}{2} D \leq 0$. Integrating this in time yields~\eqref{e:e_i2}.

We now proceed with the proof of \Cref{p.max_decay}.

\begin{proof}[Proof of \Cref{p.max_decay}]
Let
\[
	t_0 = \sup\big\{ t >0 : \sup_{s \in [0,t]} s^{2/3} M(s) < A\big\}, \quad\quad A=2C_1^{1/3},
\]
with the $C_1$ from Lemma~\ref{l.dissipation}. It is clear that if $t_0 = \infty$ then the proof is finished.  We proceed by contradiction assuming that $t_0$ is finite.  By continuity, it is also clear that $M(t_0) = A t_0^{-2/3}$.  There are two cases to consider.

{\bf Case one: $M(t_0/2) > 2 M(t_0)$.} 
Since $t_0/2 < t_0$, then, using the definition of $t_0$, we have
\[
	\frac{2A}{t_0^{2/3}}
		= 2M(t_0)
		< M(t_0/2)
		< \frac{A}{(t_0/2)^{2/3}}
		= \frac{2^{2/3} A}{t_0^{2/3}}.
\]
This is a contradiction since $2 > 2^{2/3}$.  Hence, this case cannot occur.

{\bf Case two: $M(t_0/2) \leq 2 M(t_0)$.} 
We first combine \Cref{l.dissipation} and~\eqref{e:mp} to find
\[
	M(t)
		\leq M(t_0/2) - \frac{1}{C_1} \int_{t_0/2}^t \frac{M(s)^5}{D(s)} ds.
\]
Since this is true for all $t$, it follows that
\[
	M(t_0) \leq \overline M(t_0),
\]
where $\dot{\overline M} = - C_1^{-1} \overline M^5 D^{-1}$ and $\overline M(t_0/2) = M(t_0/2)$.  Elementary calculus yields
\be\label{e:c1171}
	M(t_0)
		\leq \overline M(t_0)
		= \left(
			M(t_0/2)^{-4}
				+ \frac{4}{C_1} \int_{t_0/2}^{t_0} D(s)^{-1} ds
			\right)^{-1/4}.
\ee

Then, using (in order) Cauchy-Schwarz, \eqref{e:e_i2}, \eqref{e.E_less_M}, and the assumption that $M(t_0/2) \leq 2 M(t_0)$, we find
\be
\begin{split}
	\frac{1}{4} t_0^2
		&= \left(\int_{t_0/2}^{t_0} \sqrt{D(s)}\frac{1}{\sqrt{D(s)}} ds\right)^2
		\leq (E(t_0/2) - E(t_0)) \int_{t_0/2}^{t_0} D(s)^{-1} ds\\
		&\leq E(t_0/2) \int_{t_0/2}^{t_0} D(s)^{-1} ds
		\leq  M(t_0/2) \int_{t_0/2}^{t_0} D(s)^{-1} ds\\
		&\leq 2 M(t_0) \int_{t_0/2}^{t_0} D(s)^{-1} ds.
\end{split}
\ee
Using this inequality in~\eqref{e:c1171}, we obtain
\be
	M(t_0)
		\leq \left(
			\frac{1}{M(\frac{t_0}{2})^{4}}
				+ \frac{1}{2C_1}\frac{t_0^2}{M(t_0)}
			\right)^{-\frac14}
		\leq
			\left(
				 \frac{1}{2C_1}\frac{t_0^2}{M(t_0)}
			\right)^{-\frac14}
		= \left(2C_1 \frac{M(t_0)}{t_0^2}\right)^{\frac14}.
\ee
Re-arranging this yields
\be
	M(t_0)
		< \frac{(2C_1)^{1/3}}{t_0^{2/3}}.
\ee
However, by the construction of $t_0$, we have that $M(t_0) = A t_0^{-2/3}=2C_1^{1/3}t_0^{-2/3}$.  Hence, we have reached a contradiction, and we conclude that case two cannot occur either.

Since both cases yield a contradiction, it follows that $t_0 = \infty$, which completes the proof.
\end{proof}

It only remains to establish \Cref{l.dissipation}.  We do this now.  The idea of the proof is to re-write $M-E$ in terms of a single integral term and then use the proof of a lemma of Constantin, Kiselev, Oberman, and Ryzhik~\cite[Lemma 2]{Constantin_2000}.  This lemma was originally used to establish the key inequality in a proof of lower bounds on the speed of Fisher-KPP fronts in the presence of shear flows in a cylinder.

%
%
%

\begin{proof}[Proof of \Cref{l.dissipation}]
As time plays no role in this lemma, we omit it notationally.  First, observe that, due to~\eqref{e.mass_one}, we have
\[
	M - E
		= M \int_\R g(x) dx  - \int_\R g(x)^2 dx
		= \int_\R g(x) (M - g(x)) dx.
\]
Notice that the integrand $g(M - g)$ is nonnegative.

Since $M$ is the maximum of $g$ and $\lim_{x\to-\infty} g(x) = 0$, we can find $x_1 < x_2$ such that
\be
	g(x_1) = \frac{M}{3},
		\quad
	g(x_2) = \frac{2M}{3},
		\quad \text{and} \quad
	\frac{M}{3} \leq g(x) \leq \frac{2M}{3} \text{ for all } x \in (x_1,x_2).
\ee
Then we have that
\be\label{e:c1172}
	\frac{M}{3}
		= \int_{x_1}^{x_2} g_x\ dx
		\leq \sqrt{x_2 - x_1} \left( \int_\R |g_x|^2 dx\right)^{1/2}.
\ee
On the other hand, since $M/3 \leq g(x) \leq 2M/3$ for all $x \in (x_1,x_2)$, then $g(M-g) \geq M^2/9$ on $(x_1,x_2)$.  It follows that
\be\label{e:c1173}
	\int_\R g(x)(M - g(x)) dx
		\geq	\int_{x_1}^{x_2} g(x) (M - g(x)) dx
		\geq  \int_{x_1}^{x_2} \frac{M^2}{9} dx
		= \frac{M^2}{9} |x_2-x_1|.
\ee
After squaring~\eqref{e:c1172} and inserting~\eqref{e:c1173} into it, we find
\be
	\frac{M^2}{9}
		\leq \frac{9}{M^2} \int_\R g(x) (M - g(x)) dx
			\int_\R |g_x|^2 dx,
\ee
which yields the claim.
\end{proof}

\appendix

\section{Basics about stochastic heat equation}\label{s.she}

For the convenience of the reader, we present some standard facts about the stochastic heat equation. We will first discuss the case of the spacetime white noise in $d=1$, then that of colored noise in $d\geq1$. 

\subsection{Spacetime white noise in $d=1$} 
For readability, we define the following notation.  For any $\eps>0$ and for any $f$ such that either $0 \leq f \in C_b(\R)$ or $f(\cdot) = \delta(\cdot)$,
we define $u_{\eps,f}$ and $u_f$ as the solutions to the following equations in $d=1$:
\[
\begin{aligned}
\partial_tu_{\eps,f}(t,x)&=\tfrac12\Delta u_{\eps,f}(t,x)+\beta\,u_{\eps,f}(t,x)\eta_\eps(t,x),  \quad &u_{\eps,f}(0,x)=f(x),\\
\partial_tu_{f}(t,x)&=\tfrac12\Delta u_{f}(t,x)+\beta\,u_{f}(t,x)\eta(t,x),   \quad &u_{f}(0,x)=f(x),
\end{aligned}
\]
where $\eta$ is the spacetime white noise, $\eta_\eps$ is the spatial mollification of $\eta$, and they are time reversals of $\xi,\xi_\eps$:
\[
\eta(t,x)=\xi(-t,x), \quad\quad \eta_\eps(t,x)=\xi_\eps(-t,x).
\]

\begin{lemma}\label{l.mmbd}


(i) For any $p\geq1, T>0$, there exists $C=C(p,\beta,T)>0$ such that 
\[
 \bfE[ u_{\eps,f}(t,x)^p]+\bfE[u_{f}(t,x)^{p}] \leq C (G_t\star f(x))^p, \quad \mbox{ for all }\eps\in(0,1),t\in (0,T],x\in\R;
\]

(ii) $u_{\eps,f}(t,x)\to u_f(t,x)$ in $L^p(\Omega)$  as $\eps\to0$, uniformly in $x\in\R$ and $t$ in compact subsets of $(0,\infty)$;

(iii) For $f\in C_b(\R)$, we have $u_f(t,x)\to f(x)$ in $L^p(\Omega)$ as $t\to0$, for each $x\in\R$;

(iv) $u_f\in C((0,\infty)\times \R, L^p(\Omega))$.

\end{lemma}

\begin{proof}
The moments bounds in (i) can be found\footnote{In \cite{chen2019comparison}, the bounds are stated only for $p\geq 2$.  However, it is also shown that, in our setting, $u_f, u_{\eps,f} >0$.  The bound for $p\in(1,2)$ follows by H\"older's inequality.} in \cite[Theorem 1.7]{chen2019comparison}. The convergence in (ii) was proved in \cite[Theorem 2.2]{bertini1995stochastic}. For (iii), we write 
\[
u_f(t,x)=G_t\star f(x)+\beta\int_0^t\int_{\R}G_{t-s}(x-y)u_f(s,y)\eta(s,y)dyds.
\]
By the moments bounds in (i) and the BDG inequality, we can show the term of the stochastic integral $\int_0^t\int_{\R}G_{t-s}(x-y)u_f(s,y)\eta(s,y)dyds$ goes to zero in $L^p(\Omega)$, as $t\to0$.  Finally, (iv) comes from the standard moment estimate: for any $n\in\N$ and $s,t\in[n^{-1},n],x,y\in\R$, there exists $C=C(n,p)$ such that 
\[
\bfE[|u_f(t,x)-u_f(s,y)|^p] \leq C (|t-s|^{\frac{p}{4}}+|x-y|^{\frac{p}{2}}).
\]
The proof is complete.
\end{proof}

Recall that 
\[
\begin{aligned}
q_\eps(t,x)=\frac{\int_{\R}U_\eps(-t,x;0,y)q_0(y)dy}{\int_{\R^2} U_\eps(-t,\tilde{x};0,y)q_0(y)dyd\tilde{x}},\quad\quad q(t,x)=\frac{\int_{\R}U(-t,x;0,y)q_0(y)dy}{\int_{\R^2} U(-t,\tilde{x};0,y)q_0(y)dyd\tilde{x}},
\end{aligned}
\]
with $U_\eps,U$ solving \eqref{e.sheeps}, \eqref{e.shephi} respectively. Define 
\begin{equation}\label{e.tilderho}
\begin{aligned}
\tilde{q}_\eps(t,x)=\frac{u_{\eps,q_0}(t,x)}{ \int_{\R} u_{\eps,q_0}(t,\tilde{x})d\tilde{x}},\quad\quad \tilde{q}(t,x)=\frac{u_{q_0}(t,x)}{ \int_{\R} u_{q_0}(t,\tilde{x})d\tilde{x}}.
\end{aligned}
\end{equation}
\begin{lemma}\label{l.timereversal}
For any $\eps>0$ and $t\geq0,x\in\R$, we have 
\[
q_{\eps}(t,x)=\tilde{q}_\eps(t,x),\quad\quad q(t,x)=\tilde{q}(t,x).
\]
\end{lemma}

\begin{proof}
By the Feynman-Kac formula, we have 
\[
\begin{aligned}
\int_{\R}U_\eps(-t,x;0,y)q_0(y)dy=&\int_{\R} q_0(y)\E_y[\delta(w_t-x)e^{\beta\int_0^t\xi_\eps(-s,w_s)ds-\frac12\beta^2R_\eps(0)t}]dy\\
=&\int_{\R}q_0(y)G_t(x-y)\E_y[e^{\beta\int_0^t\xi_\eps(-s,w_s)ds-\frac12\beta^2R_\eps(0)t}\,|\,w_t=x]dy,
\end{aligned}
\]
where we recall that $\E_y$ is the expectation with respect to $w$ starting from $w_0=y$ and $G_t(\cdot)$ is the standard heat kernel. Changing variables in the exponent $s\mapsto t-s$ and using the time reversal $\eta_\eps(t,\cdot)=\xi_\eps(-t,\cdot)$, we find
\[
\begin{aligned}
\int_{\R}U_\eps(-t,x;0,y)q_0(y)dy=&\int_{\R} q_0(y)G_t(x-y)\E_y[e^{\beta \int_0^t\eta_\eps(t-s,w_{t-s})ds-\frac12\beta^2R_\eps(0)t}\,|\,w_t=x]dy\\
=&\int_{\R} q_0(y)G_t(x-y)\E_x[e^{\beta\int_0^t\eta_\eps(t-s,w_{s})ds-\frac12\beta^2R_\eps(0)t}\,|\,w_t=y]dy.
\end{aligned}
\]

On the other hand, we write 
\[
\begin{aligned}
u_{\eps,q_0}(t,x)=&\E_x[q_0(w_t)e^{\beta \int_0^t\eta_\eps(t-s,w_{s})ds-\frac12\beta^2R_\eps(0)t}]\\
=&\int_{\R} q_0(y)G_t(y-x)\E_x[e^{\beta\int_0^t\eta_\eps(t-s,w_{s})ds-\frac12\beta^2R_\eps(0)t}\,|\, w_t=y]dy.
\end{aligned}
\]
The two expressions are equal to each other, so we have $q_\eps(t,x)=\tilde{q}_\eps(t,x)$. By sending $\eps\to0$, we also have $q(t,x)=\tilde{q}(t,x)$, which completes the proof.
\end{proof}

\begin{lemma}\label{l.negativemm}
For any $p\geq1, T>0$, there exists $C=C(p,T)$ such that, for $\eps\in(0,1), t\in[0,T]$, we have
\[
\bfE\left[\big(\int_{\R^2} U_\eps(-t,\tilde{x};0,y)q_0(y)dyd\tilde{x}\big)^{-p}\right]+\bfE\left[\big(\int_{\R^2} U(-t,\tilde{x};0,y)q_0(y)dyd\tilde{x}\big)^{-p}\right]\leq C.
\]
\end{lemma}

\begin{proof}
Define $U_{\eps,\1}(-t;s,y):=\int_{\R}U_\eps(-t,\tilde{x};s,y)d\tilde{x}$, which solves 
\[
\partial_sU_{\eps,\1}=\tfrac12\Delta_y U_{\eps,\1}+\beta\,U_{\eps,\1}\xi_\eps(s,y), \quad U_{\eps,\1}(-t;-t,y)=\1.
\]
By the statistical shift-invariance of the noise, we have
\[
U_{\eps,\1}(-t;0,\cdot)\stackrel{\text{law}}{=}u_{\eps,\1}(t,\cdot),
\]
 for each $t>0$. Combining this with Lemma~\ref{l.timereversal}, a simple observation is that, for each $t>0$, 
 \begin{equation}\label{e.obs}
 \int_\R u_{\eps,\1}(t,y)q_0(y)dy\stackrel{\text{law}}{=} \int_{\R} u_{\eps,q_0}(t,\tilde{x})d\tilde{x}.
 \end{equation}
By Jensen's inequality, we have 
\[
\big(\int_\R U_{\eps,\1}(-t;0,y)q_0(y)dy\big)^{-p} \leq \int_\R U_{\eps,\1}(-t;0,y)^{-p}q_0(y)dy,
\]
which implies 
\be\label{e.c6221}
	\begin{split}
		\bfE\left[\big(\int_{\R^2} U_\eps(-t,\tilde{x};0,y)q_0(y)dyd\tilde{x}\big)^{-p}\right]
			&\leq\int_\R \bfE[ U_{\eps,\1}(-t;0,y)^{-p}]q_0(y)dy\\
			&=\bfE[ U_{\eps,\1}(-t;0,\cdot)^{-p}],
	\end{split}
\ee
where we used the stationarity of $U_{\eps,\1}(-t;0,y)$ in the $y$ variable in the last step. The negative moment bound
\[
\sup_{\eps\in(0,1),t\in[0,T]}\bfE[ u_{\eps,\1}(t,\cdot)^{-p}] \leq C(p,\beta,T)
\]
can be found, e.g., in \cite[Corollary 4.8]{hu2018asymptotics}, and hence the proof is complete.
\end{proof}

\subsection{Colored noise in $d\geq 1$}
For colored noise $\eta_\phi$, we consider the equation
\[
\begin{aligned}
&\partial_t u_f(t,x)=\tfrac12\Delta u_f(t,x)+\beta\,u_f(t,x)\eta_\phi(t,x),  \quad t>0, x\in\R^d,\\
&u_f(0,x)=f(x),
\end{aligned}
\]
where $d\geq 1$ and either $0\leq f\in C_b(\R^d)$ or $f(x)=\delta(x)$. The follow lemma holds:
\begin{lemma}\label{l.mmbdhighD}
For any $p\geq1, T>0$, there exists $C=C(d,p,\beta,T)>0$ such that 
\[
\begin{aligned}
&\bfE[u_{f}(t,x)^{p}] \leq C (G_t\star f(x))^p, \\
&\bfE[u_{\1}(t,x)^{-p}]\leq C,
\end{aligned}
\]
for all $t\in (0,T],x\in\R^d$.
\end{lemma}

The positive moment bound can be found in \cite[Theorem 1.7]{chen2019comparison}, and the negative moment bound in \cite[Corollary 4.8]{hu2018asymptotics}.

\section{Proofs of Lemmas~\ref{l.bdrhoeps}, \ref{l.conQ}, \ref{l.tzero} and \ref{l.bdQn}}
\label{s.lem}

\begin{proof}[Proof of Lemma~\ref{l.bdrhoeps}]
By Lemma~\ref{l.timereversal}, we know that 
\[
q_\eps(t,x)=\frac{u_{\eps,q_0}(t,x)}{ \int_{\R} u_{\eps,q_0}(t,\tilde{x})d\tilde{x}}
	\quad
	\text{ and }
	\quad
q(t,x)=\frac{u_{q_0}(t,x)}{ \int_{\R} u_{q_0}(t,\tilde{x})d\tilde{x}}.
\]
Applying Lemma~\ref{l.mmbd} (i), Lemma~\ref{l.negativemm} and H\"older's inequality, we derive \eqref{e.bdrhoeps}. To show the convergence of $q_\eps(t,x)\to q(t,x)$, we write 
\[
\begin{aligned}
q_\eps(t,x)-q(t,x)=&\frac{u_{\eps,q_0}(t,x)-u_{q_0}(t,x)}{\int_{\R} u_{\eps,q_0}(t,\tilde{x})d\tilde{x}}\\
&+\frac{u_{q_0}(t,x)}{\int_{\R} u_{\eps,q_0}(t,\tilde{x})d\tilde{x} \int_{\R} u_{q_0}(t,\tilde{x})d\tilde{x}}\big( \int_{\R} u_{q_0}(t,\tilde{x})d\tilde{x}-\int_{\R} u_{\eps,q_0}(t,\tilde{x})d\tilde{x}\big).
\end{aligned}
\]
The first term on the r.h.s.\ goes to zero in $L^p(\Omega)$ by Lemma~\ref{l.mmbd} (ii) and Lemma~\ref{l.negativemm}. Together with  \eqref{e.obs}, we can show the second term goes to zero similarly.
\end{proof}

\begin{proof}[Proof of Lemma~\ref{l.conQ}]
The continuity of $Q_n$ follows from \eqref{e.bdrhoeps}, H\"older's inequality and the convergence of
\begin{equation}\label{e.conq}
\bfE[|q(t_n,x_n)-q(t,x)|^p]\to0
\end{equation} if $(t_n,x_n)\to (t,x)\in(0,\infty)\times \R$ as $n\to\infty$. To prove \eqref{e.conq}, it suffices to apply Lemmas~\ref{l.mmbd} (iv), \ref{l.timereversal} and \ref{l.negativemm}. We omit the details here.
\end{proof}

\begin{proof}[Proof of Lemma~\ref{l.tzero}]
By the assumption of $q_0\in C_c(\R)$ and  Lemma~\ref{l.mmbd} (i), we know that $\E[|q(t,x)|^p]\leq C$, uniformly in $t\in(0,T],x\in\R$. Then it suffices to show that for $x\in\R$, $q(t,x)\to q_0(x)$ in $L^p(\Omega)$ as $t\to0$. We write
\[
q(t,x)-q_0(x)=\frac{u_{q_0}(t,x)-q_0(x)}{ \int_{\R} u_{q_0}(t,\tilde{x})d\tilde{x}}+\frac{q_0(x)}{ \int_{\R} u_{q_0}(t,\tilde{x})d\tilde{x}}\big(1- \int_{\R} u_{q_0}(t,\tilde{x})d\tilde{x}\big),
\]
and the first term on the r.h.s.\ goes to zero in $L^p(\Omega)$ by Lemma~\ref{l.mmbd} (iii) and Lemma~\ref{l.negativemm}. For the second term, the proof is similar  with another use of \eqref{e.obs} and the fact that $\int q_0=1$.
\end{proof}

\begin{proof}[Proof of Lemma~\ref{l.bdQn}]
The goal is to show 
\[
Q_n(t,x_1,\ldots,x_n)=\bfE[q(t,x_1)\ldots q(t,x_n)] \leq C \prod_{j=1}^n G_t(x_j)
\]
in $d\geq3$ when $\beta$ is small. By Lemma~\ref{l.timereversal}, we have 
 \[
 q(t,x)= \frac{u_{q_0}(t,x)}{\int_{\R^d} u_{q_0}(t,\tilde{x})d\tilde{x}},
 \]
 with $u_{q_0}$ the solution to the stochastic heat equation starting from $q_0$:
 \[
 \partial_t u_{q_0}=\tfrac12\Delta u_{q_0}+\beta u_{q_0}\eta_\phi, \quad\quad u_{q_0}(0,x)=q_0(x).
 \]
Recall that here we consider the initial data $q_0(x)=\delta(x)$.

By the proof of Lemma~\ref{l.negativemm} (see~\eqref{e.c6221}), we have 
\[
\bfE\Big[\Big(\int_{\R^d} u_{q_0}(t,\tilde{x})d\tilde{x}\Big)^{-p}\Big] \leq \bfE[ u_{\1}(t,\cdot)^{-p}],
\]
with $u_{\1}$ solving the equation with constant initial data 
\[
\partial_t u_{\1}=\tfrac12\Delta u_{\1}+\beta u_{\1}\eta_\phi, \quad\quad u_{\1}(0,x)=1.
 \]
By \cite[Proposition 2.3]{kpz1}, we know that for any $d\geq3$ and $p\geq1$, there exists positive constants $\beta_0,C>0$ such that if $\beta<\beta_0$, then 
\[
\sup_{t\geq0} \bfE[ u_{\1}(t,\cdot)^{-p}] \leq C.
\]
By H\"older's inequality, this implies 
\[
Q_n(t,x_1,\ldots,x_n) \leq C \prod_{j=1}^n \bfE[ u_{q_0}(t,x_j)^{2n}]^{\frac{1}{2n}}. 
\]
Thus, the proof of the lemma reduces to the following moment estimate of $u_{q_0}(t,x)$: for any $n\geq1$, there exists $\beta_0,C>0$ such that if $\beta<\beta_0$, then 
\begin{equation}\label{e.141}
\bfE[ u_{q_0}(t,x)^n]^{\frac{1}{n}} \leq CG_t(x), \quad\quad \mbox{ for all } t>0,x\in\R^d.
\end{equation}

To prove \eqref{e.141}, we first express $u_{q_0}$ by the Feynman-Kac formula  as 
\[
u_{q_0}(t,x)=\E_x[q_0(w_t)e^{\beta \int_0^t \eta_\phi(t-s,w_s)ds-\frac12\beta^2R(0)t}],
\]
with $\E_x$ the expectation with respect to the Brownian motion $w$ starting at $x$. 
By the replica method, the $n-$th moment equals
\[
\bfE[ u_{q_0}(t,x)^n]=\E_x\big[\prod_{j=1}^n q_0(w^j_t)e^{\beta^2\sum_{1\leq \ell_1<\ell_2\leq n}\int_0^t R(w^{\ell_1}_s-w^{\ell_2}_s)ds}\big],
\]
with $\{w^j\}_{j=1,\ldots,n}$ independent Brownian motions starting at $x$. By conditioning on the positions of $w^j$ at time $t$, we further write the expectation as an integral
\[
\begin{aligned}
&\bfE[ u_{q_0}(t,x)^n]\\
&=\int_{\R^{nd}}\prod_{j=1}^n q_0(y_j)G_t(x-y_j) \E_x[e^{\beta^2 \sum_{1\leq \ell_1<\ell_2\leq n}\int_0^t R(w^{\ell_1}_s-w^{\ell_2}_s)ds}\,|\,w_t^j=y_j,j=1,\ldots,n]dy_1\ldots dy_n.
\end{aligned}
\]
By \cite[Lemma A.2]{kpz1}, the following estimate holds for $\beta\ll1$:
\[
\sup_{t>0,x,y_1,\ldots,y_n\in\R^d} \E_x[e^{\beta^2 \sum_{1\leq \ell_1<\ell_2\leq n}\int_0^t R(w^{\ell_1}_s-w^{\ell_2}_s)ds}\,|\,w_t^j=y_j,j=1,\ldots,n] \leq C,
\]
which yields
\[
\bfE[ u_{q_0}(t,x)^n]\leq C \big(\int_{\R^d} q_0(y)G_t(x-y)dy\big)^n=CG_t(x)^n,
\]
and completes the proof.
\end{proof}




\end{document}